\documentclass[reqno]{amsart}%[a4paper,10pt,reqno]{amsart}

\usepackage[english]{babel}
\usepackage{amsmath,amssymb,amsthm}
\usepackage{hyperref}
\usepackage[centering]{geometry}
\usepackage{xcolor}

\renewcommand{\epsilon}{\varepsilon}            %%  changes epsilon
\newtheorem{theorem}{Theorem}[section]   %% definition of theorem environment
\newtheorem*{theorem*}{Theorem}          %% a theorem environment without numbering
\newtheorem{lemma}[theorem]{Lemma}

\theoremstyle{definition}

\newtheorem{corollary}[theorem]{Corollary}
\newtheorem{example}[theorem]{Example}

\newtheorem{remark}{Remark}[section]
\newtheorem*{acknow}{Acknowledgments}

\numberwithin{equation}{section}

\title[On Hypersurfaces of $\mathbb{H}^2\times \mathbb{H}^2$]
{On Hypersurfaces of $\mathbb{H}^2\times \mathbb{H}^2$}
\thanks{}

\author{Dong Gao \and Hui Ma \and Zeke Yao}

\address{D.~Gao, Department of Mathematics, School of Science, Beijing
University of Civil Engineering and Architecture, Beijing 102616, P.R. China}
\email{gaodong@bucea.edu.cn}

\address{H.~Ma, Department of Mathematical Sciences, Tsinghua
University, Beijing, 100084, P.R. China}
\email{ma-h@tsinghua.edu.cn}

\address{Z.~Yao, Department of Mathematical Sciences, Tsinghua
University, Beijing, 100084, P.R. China}
\email{yaozkleon@163.com}

\date{}

\begin{document}

\begin{abstract}
In this paper, we study hypersurfaces in $\mathbb{H}^2\times\mathbb{H}^2$.
We first classify the hypersurfaces with constant principal curvatures and constant product angle
function. Then, we classify homogeneous hypersurfaces and isoparametric hypersurfaces,
respectively. Finally, we classify the hypersurfaces with at most two distinct constant principal
curvatures, as well as those with three distinct constant principal curvatures under some
additional conditions.
\end{abstract}

\maketitle

%%%%%%%%%%%%%%%%%%%%%%
\section{Introduction}\label{sect:1}
%%%%%%%%%%%%%%%%%%%%%%

Let $M$ be an orientable hypersurface of Riemannian manifold $(N^n, g)$.
We say that $M$ is an isoparametric hypersurface of $N^n$, i.e.,
there exists an isoparametric function $F:N^n \rightarrow \mathbb{R}$ such that $M=F^{-1}(l)$,
for some regular value $l$ of $F$. Here $F$ is called an isoparametric function
if the gradient and the Laplacian of $F$ satisfy
$$
\|\nabla F\|^{2}=f_1(F), \quad \Delta F=f_2(F),
$$
where $f_1, f_2: \mathbb{R} \rightarrow \mathbb{R}$ are smooth functions.

When $N^n$ is a real space form, it is well known that isoparametric hypersurfaces
are equivalent to hypersurfaces with constant principal curvatures, which have been
studied extensively. We refer to the excellent survey \cite{Chi,Tho} and the references therein.
But when the ambient space has nonconstant sectional curvature, the isoparametricity of
a hypersurface is generally not equivalent to the constancy of the principal curvatures.
The first counterexamples were constructed by Wang \cite{Wang1} in complex projective spaces.
When $N^n$ is a non-flat complex space form, Ge, Tang and Yan \cite{GTY} proved that
isoparametric hypersurfaces in $\mathbb{C}P^{2n}$ are homogeneous,  but this is no longer valid for isoparametric hypersurfaces  in $\mathbb{C}P^{2n+1}$. Later, Dom\'{\i}nguez-V\'{a}zquez \cite{D} classified
isoparametric hypersurfaces in complex projective spaces 
except for the case of $\mathbb{C}P^{15}$, and D\'{\i}az-Ramos, Dom\'{\i}nguez-V\'{a}zquez and  Sanmart\'{\i}n-L\'{o}pez \cite{DDS} classified isoparametric hypersurfaces
in complex hyperbolic spaces.  Furthermore, in canonical Riemannian manifolds,
there have been some other interesting results on isoparametric
hypersurfaces (cf. \cite{DD,D-M,dos,Julio,TY} etc.),
on hypersurfaces with constant principal curvatures (cf. \cite{B,DD1,D-M,K,Ta2,Wang2} etc.), and on homogeneous hypersurfaces  (cf. \cite{BT,D-M,Ta1} etc.).
However, the study on these three types of hypersurfaces in product manifolds of the space forms is still at the beginning stage. In a recent nice paper \cite{Ur}, Urbano classified homogeneous hypersurfaces
and isoparametric hypersurfaces of $\mathbb{S}^2\times \mathbb{S}^2$, respectively.

In this paper, we study hypersurfaces of
$\mathbb{H}^2\times \mathbb{H}^2$, the Riemannian product of two hyperbolic planes of
curvature $-1$. The Riemannian product manifold $\mathbb{H}^2\times \mathbb{H}^2$,
together with $\mathbb{C}^2$ and $\mathbb{C}H^2$, are the only noncompact Hermitian symmetric
$4$-manifolds. It is also a K\"ahler-Einstein manifold. Moreover,
the complex hyperbolic quadric $Q^{2*}$ with Einstein constant $-1$ is
holomorphically isometric to the K\"ahler surface $\mathbb{H}^2\times\mathbb{H}^2$.
Due to the curvature nature of the ambient manifold,
hypersurfaces in $\mathbb{H}^2\times \mathbb{H}^2$ show more diversity.

Before stating our main results, we first recall that there is a natural  product
structure $P$ on $\mathbb{H}^2\times\mathbb{H}^2$ defined by $P(v_1,v_2):=(v_1,-v_2)$ for
any tangent vector fields $v_1, v_2$ on $\mathbb{H}^2$. Then,
for an orientable hypersurface $M$ of $\mathbb{H}^2\times\mathbb{H}^2$ with
a unit normal vector field $N$,  we can introduce an important function $C$ defined by
$C:=\langle PN, N\rangle$. Here, $\langle\cdot,\cdot\rangle$ denotes the standard product metric
on $\mathbb{H}^2\times\mathbb{H}^2$. Denote $V:=PN-CN$ the tangential part of $PN$.
The geometry of hypersurfaces in $\mathbb{H}^2\times\mathbb{H}^2$ is closely related
to this function $C$. Hereafter, for the sake of brevity, we shall call $C$ the
{\it product angle function} of $M$. Now, we study the hypersurfaces with
constant product angle function $C$. Under the assumption of constant principal curvatures,
our first main result can be stated as follows, which is fundamental and very useful
in proving our subsequent results.

\begin{theorem}\label{thm:1.1}
Let $M$ be a hypersurface of $\mathbb{H}^2\times \mathbb{H}^2$ with constant principal curvatures
and constant product angle function $C$. Then, up to isometries of
$\mathbb{H}^2\times \mathbb{H}^2$, one of the following four cases occurs:
\begin{itemize}
\item[(1)] $M$ is an open part of
$M_\Gamma$, where $\Gamma$ is a curve of $\mathbb{H}^2$ with constant curvature
(see Example \ref{exam:3.1}); or
\item[(2)] $M$ is an open part of $M_{1,-1}^{c}$ for some $c\in(0,1)$
(see Example \ref{exam:3.6}); or
\item[(3)] $M$ is an open part of $M_{1,1}^{c}$ for some $c\in(0,1)$
(see Example \ref{exam:3.7}); or
\item[(4)] $M$ is an open part of $M_\tau$ for some $\tau<-1$ (see Example \ref{exam:3.8}).
\end{itemize}
\end{theorem}

\begin{remark}\label{rem:1.1}
The hypersurfaces stated in (1), (2), (3) and (4) of Theorem \ref{thm:1.1} have constant
product angle function $C=1$, $C=1-2c$, $C=1-2c$ and $C=0$, respectively.
This is different from the situation in $\mathbb{S}^2\times\mathbb{S}^2$.
Notice that, there are only two families of hypersurfaces in $\mathbb{S}^2\times\mathbb{S}^2$
with constant principal curvatures and constant product angle function $C$,
in which one family has $C=1$ and the other family has $C=0$ (\cite{Ur}).
\end{remark}

\begin{remark}\label{rem:1.2}
In the following Example \ref{exam:3.4}, for any given $c\in(0,1)$, by using two smooth curves
of $\mathbb{H}^2$ with curvature functions $\kappa$ and $\tilde{\kappa}$, we can construct
a hypersurface $M_{\kappa,\tilde{\kappa}}^c$ with constant product angle function $C=1-2c$.
By further selecting special curves of $\mathbb{H}^2$, we obtain the hypersurfaces
$M_{1,-1}^{c}$ and $M_{1,1}^{c}$ mentioned in Theorem \ref{thm:1.1}.
\end{remark}

\begin{remark}\label{rem:1.3}
In both $\mathbb{S}^2\times\mathbb{S}^2$ and $\mathbb{H}^2\times \mathbb{H}^2$,
the hypersurfaces with constant principal curvatures and constant product angle function
$C$ are equivalent to those hypersurfaces with constant mean curvature,
constant scalar curvature and constant product angle function $C$.
The classification of such hypersurfaces in $\mathbb{S}^2\times\mathbb{S}^2$
is given in Corollary 1(3) of \cite{Ur}.
\end{remark}

Then, by a direct application of Theorem \ref{thm:1.1},
we classify homogeneous hypersurfaces in $\mathbb{H}^2\times\mathbb{H}^2$.

\begin{corollary}\label{cor:1.2}
Let $M$ be a homogeneous hypersurface of $\mathbb{H}^{2} \times \mathbb{H}^{2}$.
Then, up to isometries of $\mathbb{H}^2\times \mathbb{H}^2$, one of the following four cases occurs:
\begin{itemize}
\item[(1)] $M$ is $M_\Gamma$, where $\Gamma$ is a complete curve of $\mathbb{H}^2$
with constant curvature; or
\item[(2)] $M$ is $M_{1,-1}^{c}$ for some $c\in(0,1)$; or
\item[(3)] $M$ is $M_{1,1}^{c}$ for some $c\in(0,1)$; or
\item[(4)] $M$ is $M_\tau$ for some $\tau<-1$.
\end{itemize}
\end{corollary}

As regards isoparametric hypersurfaces, in addition to the previous definition,
there is another equivalent characterization. A hypersurface of a Riemannian manifold
is isoparametric if and only if its locally defined parallel hypersurfaces
have constant mean curvature. Based on this equivalence, we prove
that isoparametric hypersurfaces
in $\mathbb{H}^{2} \times \mathbb{H}^{2}$ must have constant principal curvatures
and constant product angle function. Thus, according to Theorem \ref{thm:1.1},
we obtain the following classification result:

\begin{theorem}\label{thm:1.3}
Let $M$ be an isoparametric hypersurface of $\mathbb{H}^{2} \times \mathbb{H}^{2}$.
Then, up to isometries of $\mathbb{H}^2\times \mathbb{H}^2$,
one of the following four cases occurs:
\begin{itemize}
\item[(1)] $M$ is an open part of
$M_\Gamma$, where $\Gamma$ is a curve of $\mathbb{H}^2$ with constant curvature; or
\item[(2)] $M$ is an open part of $M_{1,-1}^{c}$ for some $c\in(0,1)$; or
\item[(3)] $M$ is an open part of $M_{1,1}^{c}$ for some $c\in(0,1)$; or
\item[(4)] $M$ is an open part of $M_\tau$ for some $\tau<-1$.
\end{itemize}
\end{theorem}

In the following, we study the hypersurfaces of $\mathbb{H}^{2}\times\mathbb{H}^{2}$
with constant principal curvatures. Firstly, we get the classification of hypersurfaces
with at most two distinct constant principal curvatures.

\begin{theorem}\label{thm:1.4}
Let $M$ be a hypersurface of $\mathbb{H}^{2} \times \mathbb{H}^{2}$.
If $M$ has at most two distinct constant principal curvatures,
then up to isometries of $\mathbb{H}^2\times \mathbb{H}^2$,
$M$ is either an open part of $M_\Gamma$, where $\Gamma$ is a curve of $\mathbb{H}^2$
with constant curvature, or $M$ is an open part of $M_{1,-1}^{1/2}$.
\end{theorem}

\begin{remark}\label{rem:1.4}
Recall that the hypersurfaces in $\mathbb{S}^{2}\times$ $\mathbb{S}^{2}$
with at most two distinct constant principal curvatures are exactly the hypersurfaces
with constant principal curvatures and $C=1$ (\cite{Ur}).
However, in $\mathbb{H}^{2} \times \mathbb{H}^{2}$, such equivalence
is no longer true. In fact, Theorem \ref{thm:1.4} shows that,
there is one more hypersurface of $\mathbb{H}^{2} \times \mathbb{H}^{2}$
with two distinct constant principal curvatures, which has $C=0$.
\end{remark}

The classification problem of hypersurfaces with three distinct constant principal curvatures
in $\mathbb{S}^{2} \times \mathbb{S}^{2}$ is hard and still open. By studying the critical
points of the product angle function $C$, Urbano obtained a partial classification result
of hypersurfaces with three distinct constant principal curvatures, under compactness and other
curvature conditions. In $\mathbb{H}^{2} \times \mathbb{H}^{2}$, all the known examples with
constant principal curvatures are noncompact, which does not allow us to characterize these
hypersurfaces by analyzing the critical points of the function $C$.
From the following Lemma \ref{lemma:3.3}, hypersurfaces of $\mathbb{H}^{2}\times \mathbb{H}^{2}$
with $C^2=1$ have at most two distinct principal curvatures. Since we focus on
local geometry here, when studying a hypersurface $M$ with three distinct
constant principal curvatures, we always assume that the product angle function $C\neq\pm1$ on $M$,
which ensures that the vector field $V$ is nonzero on $M$. Thus, according to the projection of
the vector field $V$ in different eigenspaces, we give the following two results:

\begin{theorem}\label{thm:1.5}
Let $M$ be a hypersurface of $\mathbb{H}^{2}\times\mathbb{H}^{2}$
with three distinct constant principal curvatures. If $V$ is a principal curvature
vector field on $M$, then up to isometries of $\mathbb{H}^2\times \mathbb{H}^2$,
$M$ is either an open part of $M_{1,-1}^{c}$ for some
$c\in(0,\frac{1}{2})\cup(\frac{1}{2},1)$, or $M$ is an open part of $M_{1,1}^{c}$
for some $c\in(0,1)$, or $M$ is an open part of $M_\tau$ for some $\tau<-1$.
\end{theorem}

\begin{theorem}\label{thm:1.6}
Let $M$ be a hypersurface of $\mathbb{H}^{2} \times \mathbb{H}^{2}$ with three distinct
constant principal curvatures and Gauss-Kronecker curvature $K=0$. If $V$ has nonzero
components in at most two of eigenspaces of shape operator $A$,
then up to isometries of $\mathbb{H}^2\times \mathbb{H}^2$,
$M$ is either an open part of $M_{1,-1}^{c}$ for some
$c\in(0,\frac{1}{2})\cup(\frac{1}{2},1)$, or $M$ is an open part of $M_{1,1}^{c}$
for some $c\in(0,1)$, or $M$ is an open part of $M_\tau$ for some $\tau<-1$.
\end{theorem}

The paper is organized as follows: In Section \ref{sect:2}, we review basic properties of
$\mathbb{H}^2\times\mathbb{H}^2$ and the geometry of its hypersurfaces. In Section \ref{sect:3},
we introduce the examples $M_\Gamma$, $M_{1,-1}^{c}$, $M_{1,1}^{c}$
and $M_\tau$ which appear in Theorems \ref{thm:1.1}--\ref{thm:1.6},
and present a key characterization for the hypersurface $M_{\kappa,\tilde{\kappa}}^c$
(Theorem \ref{thm:3.9}). Finally, Sections \ref{sect:4}--\ref{sect:6} are dedicated to
the proofs of Theorems \ref{thm:1.1}--\ref{thm:1.6}, respectively.

\section{Preliminaries}\label{sect:2}
\subsection{The geometric structure on $\mathbb{H}^{2}\times \mathbb{H}^{2}$}\label{sect:2.1}
Let $\mathbb{R}_{1}^{3}$ be the three-dimensional Minkowski space with the
Lorentzian metric $\langle\cdot,\cdot\rangle$.
The hyperbolic plane of curvature $-1$ can be defined as the following subset
of $\mathbb{R}_{1}^{3}$ :
$$
\mathbb{H}^{2}=\left\{\left(x_{1}, x_{2}, x_{3}\right) \in \mathbb{R}_{1}^{3}
\mid-x_{1}^{2}+x_{2}^{2}+x_{3}^{2}=-1  , x_{1}>0\right\}.
$$
The standard complex structure $J$ on $\mathbb{H}^{2}$ is defined by
$$J_xu=x\boxtimes u,$$
for all $x\in \mathbb{H}^{2}$ and all $u\in T_x\mathbb{H}^{2}$, where $\boxtimes$
is the Lorentzian cross product defined by
$$
(a_1,a_2,a_3)\boxtimes (b_1,b_2,b_3)=(a_3b_2-a_2b_3, a_3b_1-a_1b_3, a_1b_2-a_2b_1).
$$

Throughout the paper we will consider $\mathbb{H}^2 \times  \mathbb{H}^2$
as embedded naturally in $\mathbb{R}_{1}^{3} \times \mathbb{R}_{1}^{3} \cong \mathbb{R}_{2}^{6}$,
with the induced Riemannian product metric which we also denote by $\langle\cdot,\cdot\rangle $.
We define two complex structures on $\mathbb{H}^2\times \mathbb{H}^2$ by
$$
J_1=(J,J),\quad J_2=(J,-J),
$$
which endow $\mathbb{H}^2 \times  \mathbb{H}^2$ with two structures of K\"ahler surface.
It is also clear that if $\text{Id}:\mathbb{H}^2\to \mathbb{H}^2$ is the identity map and
$\mathcal{F}:\mathbb{H}^2\to \mathbb{H}^2$ is an anti-holomorphic isometry, then
$$
(\text{Id},\mathcal{F}) : \mathbb{H}^2\times \mathbb{H}^2\to \mathbb{H}^2\times \mathbb{H}^2
$$
is a holomorphic isometry from $(\mathbb{H}^2\times \mathbb{H}^2, \langle\cdot,\cdot\rangle , J_1)$
onto $(\mathbb{H}^2\times \mathbb{H}^2, \langle\cdot,\cdot\rangle , J_2)$ (cf. \cite{To}).
The isometry group of $\mathbb{H}^{2}\times \mathbb{H}^{2}$ is
$$
\operatorname{Iso}\left(\mathbb{H}^{2}\times \mathbb{H}^{2}\right)=\left\{\left(\begin{array}{cc}
A_{1} & 0 \\
0 & A_{2}
\end{array}\right),\left(\begin{array}{cc}
0 & B_1 \\
B_{2} & 0
\end{array}\right) \mid A_{1}, A_{2}, B_1, B_{2} \in \mathrm{O}^{+}(1,2)\right\},
$$
where $\mathrm{O}^{+}(1,2)$ denotes the ortochronous Lorentz group.

The product structure $P$ on $\mathbb{H}^2\times\mathbb{H}^2$ is defined by
$P: T(\mathbb{H}^2\times\mathbb{H}^2)\rightarrow T(\mathbb{H}^2\times\mathbb{H}^2)$ such that
$$
P(v_1,v_2)=(v_1,-v_2), \quad \forall\, v_1, v_2\in T\mathbb{H}^2.
$$
Obviously, we have $P=-J_1J_2=-J_2J_1$, $P^2=\mathrm{Id}$ and
\begin{equation}\label{eqn:2.1}
\langle PX, Y\rangle=\langle X,PY\rangle, \quad \forall\,X,Y\in T(\mathbb{H}^2\times\mathbb{H}^2).
\end{equation}
Moreover, $\bar{\nabla} P=0$, where $\bar{\nabla}$ is the Levi-Civita connection on
$\mathbb{H}^2\times\mathbb{H}^2$.

The curvature tensor $\bar{R}$ of $\mathbb{H}^2\times\mathbb{H}^2$
with the Riemannian product metric is given by
\begin{align*}
\langle\bar{R}(X,Y)Z,W\rangle=-\tfrac{1}{2}&\big\{\langle Y,Z\rangle\langle X,W\rangle
-\langle X,Z\rangle\langle Y,W\rangle\\
&+\langle PY,Z\rangle\langle PX,W\rangle-\langle PX,Z\rangle\langle PY,W\rangle\big\},
\end{align*}
where $X, Y, Z, W\in T(\mathbb{H}^2\times\mathbb{H}^2)$. Thus, $\mathbb{H}^2\times\mathbb{H}^2$
is an Einstein manifold with scalar curvature $-4$ and non-positive sectional curvature.

\subsection{Hypersurfaces of $\mathbb{H}^{2}\times \mathbb{H}^{2}$}\label{sect:2.2}
Let $M$ be an orientable hypersurface of $\mathbb{H}^2\times\mathbb{H}^2$ with
$N$ a unit normal vector field. The induced metric on $M$ is still denoted as
$\langle\cdot,\cdot\rangle $. Then, with respect to the product structure $P$,
the product angle function $C: M\rightarrow\mathbb{R}$ and a vector field
$V$ tangent to $M$ are defined by
\begin{align*}
C&:=\langle PN,N\rangle=\langle J_1 N,J_2 N\rangle,\\
V&:=PN-CN.
\end{align*}
It is clear that $-1\leq C\leq 1$ and $\|V\|^2:=\langle V,V\rangle=1-C^2$.

Let $T: TM\rightarrow TM$ be the tangential component of the restriction of $P$ to $M$, i.e.,
$$
TX=P X-\langle PX,N\rangle N=P X-\langle X,V\rangle N
$$
for any tangent vector field $X$ of $M$.
Let $\nabla$ be the Levi-Civita connection of the induced metric on $M$.
The Gauss and Weingarten formulae say that
\begin{align*}
\bar{\nabla}_X Y=\nabla_X Y+\langle AX,Y\rangle N, \quad \bar{\nabla}_XN=-AX,
\end{align*}
where $A$ is the shape operator of $M$.

Now, the Gauss and Codazzi equations of $M$ are given by
\begin{equation}\label{eqn:2.2}
\begin{aligned}
R(X, Y)Z=&-\tfrac{1}{2}\big[\langle Y,Z\rangle X-\langle X,Z\rangle Y+\langle TY,Z\rangle TX
-\langle TX,Z\rangle TY\big]\\
&+\langle AY,Z\rangle AX-\langle AX,Z\rangle AY,
\end{aligned}
\end{equation}
\begin{equation}\label{eqn:2.3}
(\nabla_XA)Y-(\nabla_Y A)X=-\tfrac{1}{2}\big[\langle X,V\rangle TY-\langle Y,V\rangle TX\big],
\end{equation}
where $X, Y, Z\in TM$, and $R$ denotes the curvature tensor of $M$ with respect
to the metric $\langle\cdot,\cdot\rangle $.

It follows from \eqref{eqn:2.2} that the Ricci curvature tensor of $M$ is given by
\begin{equation}\label{eqn:2.4}
\begin{aligned}
\operatorname{Ric}(X,Y):&=\sum_{i=1}^3\langle R(X,E_i)E_i,Y\rangle\\
&=-\tfrac{1}{2}\big[\langle X,Y\rangle-C\langle TX,Y\rangle
+\langle X,V\rangle \langle Y,V\rangle \big]+H\langle AX,Y\rangle-\langle A^2X,Y\rangle,
\end{aligned}
\end{equation}
where $E_i\ (1\leq i\leq 3)$ is a local orthonormal frame field of $M$,
$H=\operatorname{tr}A$ denotes the mean curvature of $M$, and we have used
$$
\sum_{i=1}^{3}\langle PE_i,E_i\rangle=\operatorname{tr}P-\langle PN,N\rangle=-C,
$$
$$
\sum_{i=1}^{3}\langle PX,E_i\rangle\langle PY,E_i\rangle=\langle PX,PY\rangle
-\langle PX,N\rangle\langle PY,N\rangle=\langle X,Y\rangle-\langle X,V\rangle\langle Y,V\rangle.
$$

Thus the Ricci curvature along a vector $X$ is given by
\begin{equation}\label{eqn:2.5}
{\rm Ric}(X, X)=-\tfrac{1}{2}\big[\langle X,X\rangle-C\langle TX,X\rangle
+\langle X,V\rangle^2\big]+H\langle AX,X\rangle-\langle A^2X,X\rangle.
\end{equation}
Then it follows that the scalar curvature $\rho$ of $M$ is given by
\begin{equation}\label{eqn:2.6}
\rho=-2+H^{2}-\|A\|^{2}.
\end{equation}

Notice that the product structure $P$ of $\mathbb{H}^2\times\mathbb{H}^2$ satisfies
\eqref{eqn:2.1} and $\bar{\nabla}P=0$. Then, we can obtain the following Lemma \ref{lemma:2.1}
which describes some properties of the function $C$ and the vector field $V$.

\begin{lemma}\label{lemma:2.1}
Let $M$ be an orientable hypersurface of $\mathbb{H}^2\times\mathbb{H}^2$
and $A$ the shape operator associated to the unit normal field $N$.
Then the gradient of $C$ and the covariant derivative of $V$ are given by
\begin{equation}\label{eqn:2.7}
\nabla C=-2 AV, \quad \nabla_{X} V=C AX -T A X, \ \ \forall\ X\in TM.
\end{equation}
\end{lemma}
\begin{proof}
By the definition of the product angle function $C=\langle PN,N\rangle$, we have
$$
\begin{aligned}
XC&=X\langle PN,N\rangle=\langle P\bar{\nabla}_X N,N\rangle+\langle PN,\bar{\nabla}_X N\rangle\\
&=-2\langle AX,V\rangle=-2\langle AV,X\rangle, \ \ \forall\ X\in TM.
\end{aligned}
$$
It follows that $\nabla C=-2 AV$.

Then, by using $V=PN-CN$ and $\nabla C=-2 AV$, we have
$$
\begin{aligned}
\nabla_X V&=\bar{\nabla}_X V-\langle AX,V\rangle N=\bar{\nabla}_X (PN-CN)-\langle AX,V\rangle N\\
&=P\bar{\nabla}_X N-(XC)N-C\bar{\nabla}_X N-\langle AX,V\rangle N\\
&=-PAX+CAX+\langle AX,V\rangle N=C AX -T A X,
\end{aligned}
$$
for any tangent vector field $X$ of $M$.
\end{proof}

\begin{remark}\label{rem:2.1}
According to Lemma \ref{lemma:2.1}, if $C$ is constant on $M$ and $C\neq \pm1$,
then $V$ is a principal curvature vector field of $M$, and it satisfies $AV=0$.
\end{remark}

\section{Examples}\label{sect:3}

In this section, we introduce some canonical examples of hypersurfaces in
$\mathbb{H}^2\times\mathbb{H}^2$. First of all, for any smooth curve of $\mathbb{H}^2$,
there is a hypersurface in $\mathbb{H}^2\times\mathbb{H}^2$ with $C=1$.

\begin{example}\label{exam:3.1}
For any smooth curve $\Gamma$ of $\mathbb{H}^2$, one can define a hypersurface
in $\mathbb{H}^2\times\mathbb{H}^2$ by
$$
M_\Gamma:=\left\{(x,y)\in \mathbb{H}^2\times\mathbb{H}^2
~|~x\in \Gamma,\ y\in\mathbb{H}^2\right\}.
$$
\end{example}

Let $\kappa_{\Gamma}$ be the curvature of $\Gamma$ in $\mathbb{H}^2$.
Denote the unit normal vector field of $\Gamma$ in $\mathbb{H}^2$ by $N$. Then the unit normal vector field of $M_\Gamma$ is $(N,0)$. It is obvious that  the hypersurface $M_\Gamma$ has constant product angle function $C=1$, and its principal curvatures are $\kappa_{\Gamma}$, $0$ and $0$.
On the other hand, from the expression of geodesics of $\mathbb{H}^2\times\mathbb{H}^2$, we know that
the parallel hypersurface of $M_\Gamma$ at distance $l$ is $\tilde{\Gamma}\times \mathbb{H}^2$,
where $\tilde{\Gamma}$ is a parallel curve of $\Gamma$ at distance $l$
in $\mathbb{H}^2$. Based on the above facts, we easily get the following lemma:

\begin{lemma}\label{lemma:3.2}
For any hypersurface $M_\Gamma$ of $\mathbb{H}^2\times\mathbb{H}^2$,
the following four statements are equivalent:
\begin{enumerate}
\item[(1)]
$M_\Gamma$ is an open part of a homogeneous hypersurface;

\item[(2)]
$M_\Gamma$ has constant principal curvatures;

\item[(3)]
$M_\Gamma$ is an isoparametric hypersurface;

\item[(4)]
$\Gamma$ is an open part of a complete curve in $\mathbb{H}^2$ with constant curvature.
\end{enumerate}
\end{lemma}

Notice that the isometry of $\mathbb{H}^2\times\mathbb{H}^2$
given by $(p,q)\mapsto(q,p)$ transforms the hypersurfaces with constant $C=-1$
onto the hypersurfaces with constant $C=1$. In \cite{Ur}, by studying the properties
of two complex structures and the product structure,
Urbano obtained a characterization for hypersurfaces in $\mathbb{S}^2\times\mathbb{S}^2$
with constant product angle function $C=1$.
For any hypersurface $M$ of $\mathbb{H}^2\times\mathbb{H}^2$ with $C=1$,
we decompose the tangent bundle of $M$ as $T M=\{J_1 N\}\oplus \mathfrak{D}$, where $\mathfrak{D}$ is the $2$-dimensional distribution orthogonal to $J_1 N$. Then, by using the properties
of two complex structures and the product structure of $\mathbb{H}^2\times\mathbb{H}^2$, we can prove that $\mathfrak{D}$ is a totally geodesic foliation on $M$, and the integral manifold of $\mathfrak{D}$ is an open part of $\mathbb{H}^2$.
Thus, correspondingly, we derive the following characterization for hypersurfaces in
$\mathbb{H}^2\times\mathbb{H}^2$ with $C^2=1$.

\begin{lemma}\label{lemma:3.3}
Let $M$ be a hypersurface of $\mathbb{H}^2\times\mathbb{H}^2$ with $C^2=1$.
Then, up to isometries of $\mathbb{H}^2\times\mathbb{H}^2$, $M$ is
an open part of hypersurface $M_\Gamma$ for some smooth curve $\Gamma$ in $\mathbb{H}^2$.
\end{lemma}

Next, for any two smooth curves of $\mathbb{H}^2$, we can construct a hypersurface
of $\mathbb{H}^2\times\mathbb{H}^2$ with constant product angle function $C$.

\begin{example}\label{exam:3.4}
For any constant $0<c<1$, and any smooth curves $\gamma(r)$ and $\tilde{\gamma}(s)$
of $\mathbb{H}^2$ with $r,s$ being their arc length parameters, a hypersurface in
$\mathbb{H}^2\times\mathbb{H}^2$ can be defined by the immersion
$\Phi:\Omega\subset\mathbb{R}^3\rightarrow\mathbb{H}^2\times\mathbb{H}^2$:
$(t,r,s)\rightarrow(p(t,r),q(t,s))$, where
\begin{equation}\label{eqn:3.1}
\begin{aligned}
&p(t,r)=\cosh(\sqrt{c}t)\gamma(r)+\sinh(\sqrt{c}t)N(r),\\
&q(t,s)=\cosh(\sqrt{1-c}t)\tilde{\gamma}(s)+\sinh(\sqrt{1-c}t)\tilde{N}(s),
\end{aligned}	
\end{equation}
 $N(r)$ and $\tilde{N}(s)$ are unit normal vector fields of $\gamma(r)$ and
$\tilde{\gamma}(s)$ in $\mathbb{H}^2$, respectively.
Let $\kappa(r)$ and $\tilde{\kappa}(s)$ be the curvatures of
curves $\gamma(r)$ and $\tilde{\gamma}(s)$ in $\mathbb{H}^2$, respectively.
We call such hypersurface constructed by \eqref{eqn:3.1} as $M_{\kappa,\tilde{\kappa}}^c$.

Since $\{\gamma(r), \frac{d\gamma(r)}{dr}, N(r)\}$ and $\{\tilde{\gamma}(s),
\frac{d\tilde{\gamma}(s)}{ds}, \tilde{N}(s)\}$ are two orthonormal frames
of $\mathbb{R}_1^{3}$ along curves $\gamma(r)$ and $\tilde{\gamma}(s)$ respectively, it follows that
\begin{equation}\label{eqn:3.2}
\begin{aligned}
&\frac{d^2\gamma(r)}{dr^2}=\gamma(r)+\kappa(r)N(r),\ \
\frac{dN(r)}{dr}=-\kappa(r)\frac{d\gamma(r)}{dr},\\
&\frac{d^2\tilde{\gamma}(s)}{ds^2}=\tilde{\gamma}(s)+\tilde{\kappa}(s)\tilde{N}(s),\ \
\frac{d\tilde{N}(s)}{ds}=-\tilde{\kappa}(s)\frac{d\tilde{\gamma}(s)}{ds}.
\end{aligned} 	
\end{equation}

Now, we have an orthonormal frame field $\left\{E_{1}, E_{2}, E_{3}\right\}$
of $M_{\kappa,\tilde{\kappa}}^c$ defined as
\begin{equation*}
\begin{aligned}
E_1=&\frac{\partial}{\partial t}(p(t,r),q(t,s))=(\sqrt{c}\sinh(\sqrt{c}t)\gamma(r)
+\sqrt{c}\cosh(\sqrt{c}t)N(r),\\
&\sqrt{1-c}\sinh(\sqrt{1-c}t)\tilde{\gamma}(s)
+\sqrt{1-c}\cosh(\sqrt{1-c}t)\tilde{N}(s)),\\
E_2=&\frac{1}{\cosh(\sqrt{c}t)-\sinh(\sqrt{c}t)\kappa(r)}
\frac{\partial}{\partial r}(p(t,r),q(t,s))
=(\frac{d\gamma(r)}{dr},0),\\
E_3=&\frac{1}{\cosh(\sqrt{1-c}t)-\sinh(\sqrt{1-c}t)\tilde{\kappa}(s)}
\frac{\partial }{\partial s}(p(t,r),q(t,s))
=(0,\frac{d\tilde{\gamma}(s)}{ds}).
\end{aligned}	
\end{equation*}
The unit normal vector field $N$ of $M_{\kappa,\tilde{\kappa}}^c$ is given by
\begin{equation}\label{eqn:3.3}
\begin{aligned}
N&=(N_1,N_2)=(\sqrt{\frac{1-c}{c}}\frac{\partial p(t,r)}{\partial t},
-\sqrt{\frac{c}{1-c}}\frac{\partial q(t,s)}{\partial t}),\\
&=(\sqrt{1-c}\sinh(\sqrt{c}t)\gamma(r)
+\sqrt{1-c}\cosh(\sqrt{c}t)N(r),\\
&-\sqrt{c}\sinh(\sqrt{1-c}t)\tilde{\gamma}(s)
-\sqrt{c}\cosh(\sqrt{1-c}t)\tilde{N}(s)).
\end{aligned}
\end{equation}
It follows that $M_{\kappa,\tilde{\kappa}}^c$ has constant $C=\langle PN,N\rangle=1-2c$.
By direct calculations, we get
\begin{equation}\label{eqn:3.4}
\begin{aligned}
&AE_1=0,\ \ AE_2= -\sqrt{1-c}\frac{\sinh(\sqrt{c}t)-\cosh(\sqrt{c}t)\kappa(r)}
{\cosh(\sqrt{c}t)-\sinh(\sqrt{c}t)\kappa(r)}E_2,\\
&AE_3= \sqrt{c}\frac{\sinh(\sqrt{1-c}t)-\cosh(\sqrt{1-c}t)\tilde{\kappa}(s)}
{\cosh(\sqrt{1-c}t)-\sinh(\sqrt{1-c}t)\tilde{\kappa}(s)}E_3.
\end{aligned}
\end{equation}

Next, we characterize the properties of $M_{\kappa,\tilde{\kappa}}^c$.
\begin{lemma}\label{lemma:3.5}
Let $M$ be an open part of hypersurface $M_{\kappa,\tilde{\kappa}}^c$ which constructed by
\eqref{eqn:3.1}, then
\begin{enumerate}
\item[(1)]
$M$ is minimal if and only if $c=\frac{1}{2}$, and $\kappa(r)$ and $\tilde{\kappa}(s)$
are the same constant.

\item[(2)]
$M$ has constant sectional curvature if and only if $c=\frac{1}{2}$, and $\kappa(r)$ and
$\tilde{\kappa}(s)$ are constant which satisfies $\kappa(r) \tilde{\kappa}(s)=1$.
Moreover, the sectional curvature of $M$ is $-\frac{1}{2}$.
\end{enumerate}
\end{lemma}
\begin{proof}
By \eqref{eqn:3.4}, hypersurface $M_{\kappa,\tilde{\kappa}}^c$ is minimal if and only if it holds
\begin{equation*}
\begin{aligned}
0&=\sqrt{c}(\sinh{\sqrt{1 - c}t}
-\cosh(\sqrt{1 - c}t)\tilde{\kappa}(s))(\cosh(\sqrt{c}t)
-\sinh(\sqrt{c}t)\kappa(r))\\
&\ \ \ \ -\sqrt{1-c}(\sinh(\sqrt{c}t)-\cosh(\sqrt{c}t)\kappa(r))(\cosh(\sqrt{1-c}t)
-\sinh(\sqrt{1-c}t)\tilde{\kappa}(s)).
\end{aligned}
\end{equation*}
When $c\neq\frac{1}{2}$, functions $\cosh(\sqrt{1-c}t)\cosh(\sqrt{c}t)$, $\cosh(\sqrt{1-c}t)\sinh(\sqrt{c}t)$,
$\sinh(\sqrt{1-c}t)\cosh(\sqrt{c}t)$ and $\sinh(\sqrt{1-c}t)\sinh(\sqrt{c}t)$
are linearly independent. By using this fact, we can derive that $M_{\kappa,\tilde{\kappa}}^c$
is minimal if and only if $c=\frac{1}{2}$.
Moreover, by using the independence of parameters $r$ and $s$,
we know that $\kappa(r)$ and $\tilde{\kappa}(s)$ are the same constant.

On the other hand, it follows from ${\dim M}=3$ that $M_{\kappa,\tilde{\kappa}}^c$ has constant sectional curvature if and only if $M_{\kappa,\tilde{\kappa}}^c$ is an Einstein manifold. By calculating ${\rm Ric}(E_i, E_j)$ with \eqref{eqn:2.4}  in terms of the frame $\{E_1,E_2,E_3\}$,
we derive that $M_{\kappa,\tilde{\kappa}}^c$ has constant sectional curvature if and only if $c=\frac{1}{2}$ and the sectional curvature of $M_{\kappa,\tilde{\kappa}}^c$ is $-\frac{1}{2}$. Moreover, by using the independence of parameters $r$ and $s$, we know that $\kappa(r)$ and
$\tilde{\kappa}(s)$ are constant which satisfy $\kappa(r) \tilde{\kappa}(s)=1$.
Thus we complete the proof.
\end{proof}

Since $\|N_{1}\|^{2}=1-c$ and $\|N_{2}\|^{2}=c$, the nearby parallel hypersurfaces
$\Phi_{l}: M_{\kappa,\tilde{\kappa}}^c \rightarrow \mathbb{H}^{2} \times \mathbb{H}^{2}$ of
$M_{\kappa,\tilde{\kappa}}^c$ in direction $N$ at distance $l\in(-\epsilon,\epsilon)$ are given by
\begin{align*}
\Phi_l(M_{\kappa,\tilde{\kappa}}^c)&= (\exp_p(lN_1),\exp_q(lN_2))\\
& =(\cosh \left(\sqrt{1-c} l\right) p+\left(\tfrac{1}{\sqrt{1-c}}\right)
\sinh \left(\sqrt{1-c} l\right) N_{1},\cosh \left(\sqrt{c} l\right)
q+\left(\tfrac{1}{\sqrt{c}}\right) \sinh \left(\sqrt{c} l\right) N_{2})\\
&=(\cosh(\sqrt{c}t+\sqrt{1-c}l)\gamma(r)+\sinh(\sqrt{c}t+\sqrt{1-c}l)N(r),\\
& \ \ \ \ \cosh(\sqrt{1-c}t-\sqrt{c}l)\tilde{\gamma}(s)+\sinh(\sqrt{1-c}t-\sqrt{c}l)\tilde{N}(s)),
\ \ (p,q)\in M_{\kappa,\tilde{\kappa}}^c,
\end{align*}
where $\exp$ denotes the exponential map in $\mathbb{H}^{2}$.
Direct calculations show that the principal curvatures of the parallel hypersurface
$\Phi_l(M_{\kappa,\tilde{\kappa}}^c)$ are
\begin{equation}\label{eqn:3.5}
\begin{aligned}
&\lambda_{1}=0,\quad\lambda_{2}
=-\sqrt{1-c}\frac{\sinh(\sqrt{c}t+\sqrt{1-c}l)-\cosh(\sqrt{c}t+\sqrt{1-c}l)\kappa(r)}
{\cosh(\sqrt{c}t+\sqrt{1-c}l)-\sinh(\sqrt{c}t+\sqrt{1-c}l)\kappa(r)},\\
&\lambda_3= \sqrt{c}\frac{\sinh(\sqrt{1-c}t-\sqrt{c}l)
-\cosh(\sqrt{1-c}t-\sqrt{c}l)\tilde{\kappa}(s)}
{\cosh(\sqrt{1-c}t-\sqrt{c}l)-\sinh(\sqrt{1-c}t-\sqrt{c}l)\tilde{\kappa}(s)}.
\end{aligned}
\end{equation}
\end{example}

\begin{remark}\label{rem:3.1}
According to the expressions  \eqref{eqn:3.4} of  principal curvatures of
$M_{\kappa,\tilde{\kappa}}^c$, we see that $M_{\kappa,\tilde{\kappa}}^c$
has constant principal curvatures if and only if
\begin{equation}\label{eqn:3.6}
\begin{aligned}
&-(\kappa(r)+\alpha_1)\cosh(\sqrt{c}t)+(1+\alpha_1\kappa(r))\sinh(\sqrt{c}t)=0,\\
&-(\tilde{\kappa}(s)+\alpha_2)\cosh(\sqrt{1-c}t)+(1+\alpha_2\tilde{\kappa}(s))\sinh(\sqrt{1-c}t)=0,
\end{aligned}
\end{equation}
where $\alpha_1$ and $\alpha_2$ are constant on $M_{\kappa,\tilde{\kappa}}^c$.
From the independence of parameters $t,r$ and $s$, equations in \eqref{eqn:3.6} are equivalent to
$\kappa(r)=\tilde{\kappa}(s)=\pm1$ or $\kappa(r)=-\tilde{\kappa}(s)=\pm1$.
Note that, hypersurfaces $M_{-1,-1}^c$ and $M_{-1,1}^c$ are congruent to hypersurfaces
$M_{1,1}^c$ and $M_{1,-1}^c$, respectively. Therefore, we know that
$M_{1,1}^c$ and $M_{1,-1}^c$ are the only hypersurfaces with constant principal curvatures
among $M_{\kappa,\tilde{\kappa}}^c$.
\end{remark}

\begin{remark}\label{rem:3.2}
We point out that the idea of constructing hypersurface $M_{\kappa,\tilde{\kappa}}^c$
can be applied to Riemannian product manifold $\mathbb{Q}_{c_1}^{n_1}\times \mathbb{Q}_{c_2}^{n_2}$,
where $\mathbb{Q}_{c_1}^{n_1}$ and $\mathbb{Q}_{c_2}^{n_2}$ are real space forms with constant
sectional curvatures $c_1$ and $c_2$. In fact, for any constant $c\in(0,1)$ and two smooth
hypersurfaces $M_1$ and $M_2$ of $\mathbb{Q}_{c_1}^{n_1}$ and $\mathbb{Q}_{c_2}^{n_2}$,
respectively, we can construct a hypersurface $\hat{M}$ of
$\mathbb{Q}_{c_1}^{n_1}\times \mathbb{Q}_{c_2}^{n_2}$ as follows:
$$
\hat{M}=(\tilde{\exp}_{p}(\sqrt{c}t N_1(p)),\bar{\exp}_{q}(\sqrt{1-c}t N_2(q))),
\ \ \forall\ (p,q)\in M_1\times M_2\hookrightarrow \mathbb{Q}_{c_1}^{n_1}
\times \mathbb{Q}_{c_2}^{n_2},
$$
where $\tilde{\exp}$ and $\bar{\exp}$ denote the exponential map in
$\mathbb{Q}_{c_1}^{n_1}$ and $\mathbb{Q}_{c_2}^{n_2}$, $N_1$ and $N_2$ are
unit normal vector fields of $M_1\hookrightarrow\mathbb{Q}_{c_1}^{n_1}$ and
$M_2\hookrightarrow\mathbb{Q}_{c_2}^{n_2}$, respectively.
\end{remark}

In the following, we describe hypersurfaces $M_{1,1}^c$
and $M_{1,-1}^c$ in detail, which have constant principal curvatures.

\begin{example}\label{exam:3.6}
For any given $0<c<1$, we choose $\gamma(r)$ and $\tilde{\gamma}(s)$ to be horocycle
$\{x=(x_1,x_2,x_3)\in \mathbb{H}^2|-x_1+x_3=-1\}$. Let $r$ and $s$ be arc length parameters
of $\gamma(r)$ and $\tilde{\gamma}(s)$, i.e.
$$
\gamma(r)=\{x=(\tfrac{2+r^2}{2},r,\tfrac{r^2}{2})\in \mathbb{H}^2|\ r\in(-\infty,\infty)\},\
\tilde{\gamma}(s)=\{x=(\tfrac{2+s^2}{2},s,\tfrac{s^2}{2})\in \mathbb{H}^2|\ s\in(-\infty,\infty)\}.
$$
Let $N(r)$ and $\tilde{N}(s)$ be their unit normal vector fields defined by
$$
N(r)=(\tfrac{-r^2}{2},-r,\tfrac{2-r^2}{2}),\
\tilde{N}(s)=(\tfrac{s^2}{2},s,\tfrac{-2+s^2}{2}).
$$
Then the curvatures of $\gamma(r)$ and $\tilde{\gamma}(s)$ are $\kappa(r)=1$ and
$\tilde{\kappa}(s)=-1$, respectively. In this case, we get the hypersurface $M_{1,-1}^{c}$.

Now, by \eqref{eqn:3.4}, the principal curvatures of hypersurface $M_{1,-1}^{c}$
are $0,\sqrt{1-c}$ and $\sqrt{c}$. When $c=\tfrac{1}{2}$, hypersurface $M_{1,-1}^{1/2}$
has two distinct constant principal curvatures $0$ and $\frac{1}{\sqrt{2}}$.

By \eqref{eqn:3.5}, we know that the nearby parallel hypersurfaces $\Phi_l(M_{1,-1}^{c})$
of $M_{1,-1}^{c}$ in direction $N$ at distance $l\in(-\epsilon,\epsilon)$
have the same constant principal curvatures and constant mean curvature as $M_{1,-1}^{c}$.
Thus, for any $0<c<1$, $M_{1,-1}^{c}$ is an isoparametric hypersurface.

Denote $\tilde{t}=e^{-\sqrt{c}t}$ and $\bar{t}=e^{\sqrt{1-c}t}$. We give
a subgroup of ${\rm Iso}\left(\mathbb{H}^{2}\times \mathbb{H}^{2}\right)$
by the map $G: \mathbb{R}^3\rightarrow
{\rm Iso}\left(\mathbb{H}^{2}\times \mathbb{H}^{2}\right)$,
$$
G(t,r,s)=\left(\begin{array}{cc}
G_1(t,r) & 0 \\
0 & G_2(t,s)
\end{array}\right),
$$
where
$$
G_1(t,r)=\left(\begin{array}{cccc}
\frac{1+\tilde{t}^2+r^2\tilde{t}^2}{2\tilde{t}} & r &
\frac{1-\tilde{t}^2-r^2\tilde{t}^2}{2\tilde{t}} \\
r \tilde{t} & 1 & -r \tilde{t} \\
\frac{1+r^2\tilde{t}^2-\tilde{t}^2}{2\tilde{t}} & r &
\frac{1-r^2\tilde{t}^2+\tilde{t}^2}{2\tilde{t}}
\end{array}\right)
\ \ {\rm and} \ \
G_2(t,s)=\left(\begin{array}{cccc}
\frac{1+\bar{t}^2+s^2\bar{t}^2}{2\bar{t}} & s &
\frac{1-\bar{t}^2-s^2\bar{t}^2}{2\bar{t}} \\
s \bar{t} & 1 & -s \bar{t} \\
\frac{1+s^2\bar{t}^2-\bar{t}^2}{2\bar{t}} & s &
\frac{1-s^2\bar{t}^2+\bar{t}^2}{2\bar{t}}
\end{array}\right).
$$

Then it can be directly checked that $G$ is a closed subgroup of
${\rm Iso}\left(\mathbb{H}^{2} \times \mathbb{H}^{2}\right)$.
Let $p=(1,0,0)^\top$, then hypersurface $M_{1,-1}^{c}$
is the orbit of subgroup $G$ which pass through $(p,p)$.
Hence, for any $0<c<1$,  $M_{1,-1}^{c}$ is a homogeneous hypersurface.
\end{example}

\begin{example}\label{exam:3.7}
For any given $0<c<1$, we still choose $\gamma(r)$ and $\tilde{\gamma}(s)$ to be horocycle
$\{x=(x_1,x_2,x_3)\in \mathbb{H}^2|-x_1+x_3=-1\}$. Let $r$ and $s$ be arc length parameters
of $\gamma(r)$ and $\tilde{\gamma}(s)$, i.e.
$$
\gamma(r)=\{x=(\tfrac{2+r^2}{2},r,\tfrac{r^2}{2})\in \mathbb{H}^2|\ r\in(-\infty,\infty)\},\
\gamma(s)=\{x=(\tfrac{2+s^2}{2},s,\tfrac{s^2}{2})\in \mathbb{H}^2|\ s\in(-\infty,\infty)\}.
$$
Let $N(r)$ and $\tilde{N}(s)$ be their unit normal vector fields defined by
$$
N(r)=(\tfrac{-r^2}{2},-r,\tfrac{2-r^2}{2}),\
\tilde{N}(s)=(\tfrac{-s^2}{2},-s,\tfrac{2-s^2}{2}).
$$
Then the curvatures of $\gamma(r)$ and $\tilde{\gamma}(s)$ are $\kappa(r)=\tilde{\kappa}(s)=1$.
In this case, we get the hypersurface $M_{1,1}^{c}$.

Now, by \eqref{eqn:3.4}, hypersurface $M_{1,1}^{c}$ has three distinct constant principal
curvatures $0,\sqrt{1-c}$ and $-\sqrt{c}$. By Lemma \ref{lemma:3.5}, hypersurface
$M_{1,1}^{1/2}$ is a minimal hypersurface with constant sectional curvature $-\frac{1}{2}$.

By \eqref{eqn:3.5}, we know that the nearby parallel hypersurfaces
$\Phi_l(M_{1,1}^{c})$ of $M_{1,1}^{c}$ in direction $N$ at distance $l\in(-\epsilon,\epsilon)$
have the same constant principal curvatures and constant mean curvature as
$M_{1,1}^{c}$. Thus, for any $0<c<1$, $M_{1,1}^{c}$ is an isoparametric hypersurface.

Denote $\tilde{t}=e^{-\sqrt{c}t}$ and $\hat{t}=e^{-\sqrt{1-c}t}$. We give
a subgroup of ${\rm Iso}\left(\mathbb{H}^{2}\times \mathbb{H}^{2}\right)$
by the map $B: \mathbb{R}^3\rightarrow
{\rm Iso}\left(\mathbb{H}^{2}\times \mathbb{H}^{2}\right)$,
$$
B(t,r,s)=\left(\begin{array}{cc}
B_1(t,r) & 0 \\
0 & B_2(t,s)
\end{array}\right),
$$
where
$$
B_1(t,r)=\left(\begin{array}{cccc}
\frac{1+\tilde{t}^2+r^2\tilde{t}^2}{2\tilde{t}} & r &
\frac{1-\tilde{t}^2-r^2\tilde{t}^2}{2\tilde{t}} \\
r \tilde{t} & 1 & -r \tilde{t} \\
\frac{1+r^2\tilde{t}^2-\tilde{t}^2}{2\tilde{t}} & r &
\frac{1-r^2\tilde{t}^2+\tilde{t}^2}{2\tilde{t}}
\end{array}\right)
\ \ {\rm and} \ \
B_2(t,s)=\left(\begin{array}{cccc}
\frac{1+\hat{t}^2+s^2\hat{t}^2}{2\hat{t}} & s &
\frac{1-\hat{t}^2-s^2\hat{t}^2}{2\hat{t}} \\
s \hat{t} & 1 & -s \hat{t} \\
\frac{1+s^2\hat{t}^2-\hat{t}^2}{2\hat{t}} & s &
\frac{1-s^2\hat{t}^2+\hat{t}^2}{2\hat{t}}
\end{array}\right).
$$

Then it can be directly checked that $B$ is a closed subgroup of
${\rm Iso}\left(\mathbb{H}^{2} \times \mathbb{H}^{2}\right)$.
Let $p=(1,0,0)^\top$, then hypersurface $M_{1,1}^{c}$
is the orbit of subgroup $B$ which pass through $(p,p)$.
Hence, for any $0<c<1$, $M_{1,1}^{c}$ is a homogeneous hypersurface.
\end{example}

Finally, inspired by \cite{Ur}, we construct the following example
which also appeared in \cite{GJJ}.

\begin{example}\label{exam:3.8}
For any given $\tau<-1$, we define $M_\tau=\{(p,q)\in
\mathbb{H}^2\times\mathbb{H}^2|\ \langle p,q\rangle=\tau\}$.
The unit normal vector field to $M_\tau$ in $\mathbb{H}^2\times\mathbb{H}^2$ is given by
$$N_{(p, q)}=\frac{1}{\sqrt{2(\tau^2-1)}}(q+\tau p, p+\tau q).$$
It follows that the product angle function $C=\langle PN,N\rangle=0$ holds on $M_\tau$, and we have
$$
\begin{aligned}
J_1N&=\frac{1}{\sqrt{2(\tau^2-1)}}(p\boxtimes q, q\boxtimes p),\ \
J_2N=\frac{1}{\sqrt{2(\tau^2-1)}}(p\boxtimes q, -q\boxtimes p),\\
V&=PN=\frac{1}{\sqrt{2(\tau^2-1)}}(q+\tau p, -p-\tau q).
\end{aligned}
$$
For any $(v_1,v_2)\in T_{(p,q)}M_\tau$, the shape operator $A$ associated to $N$ takes the form
$$
A(v_1,v_2)=\frac{1}{\sqrt{2(\tau^2-1)}}\big[-(v_2,v_1)-\tau(v_1,v_2)
-\langle p,v_2\rangle(p,-q)\big].
$$
Since the Lorentzian cross product is a bilinear operation satisfying
the familiar properties of a cross product with respect to the Lorentzian metric
$\langle\cdot,\cdot\rangle $, namely
$$
a\boxtimes b= -b\boxtimes a, \quad \langle a,a\boxtimes b\rangle=\langle b,a \boxtimes a\rangle=0,
$$
for all $a, b\in \mathbb{R}_1^3$. It follows that
$$
A(J_1N)=\sqrt{\frac{\tau-1}{2(\tau+1)}}J_1N,
\quad A(J_2N)=\sqrt{\frac{\tau+1}{2(\tau-1)}}J_2N,\quad AV=0.
$$

The tube of radius $l$ over the diagonal surface
$\{(p,p)\in{\mathbb{H}^2\times{\mathbb{H}^2}}\}$
is given by the sets of points $\{(x,y)\in{\mathbb{H}^2\times\mathbb{H}^2}\}$ such that
$$
(x, y)=\left(\cosh \left(\tfrac{l}{\sqrt{2}}\right) p+\sqrt{2} \sinh \left(\tfrac{l}{\sqrt{2}}
\right) v, \cosh \left(\tfrac{l}{\sqrt{2}}\right) p-\sqrt{2} \sinh \left(\tfrac{l}{\sqrt{2}}
\right) v\right),
$$
where $p \in \mathbb{H}^{2}$, $v \in T_{p} \mathbb{H}^{2}$, and $\|v\|=\tfrac{1}{\sqrt{2}}$.
From the fact that $\langle x, y\rangle=-\cosh (\sqrt{2} l)$, it follows that
the hypersurface $M_\tau$ is a tube of radius $\frac{1}{\sqrt{2}}{\rm arccosh}(-\tau)$
over the diagonal surface. It also means that $\{M_\tau,\ \tau<-1\}$ themselves are
tubes over each other, and their focal submanifold is the diagonal surface
$\{(p,p)\in{\mathbb{H}^2\times{\mathbb{H}^2}}\}$.
Thus, hypersurfaces $\{M_\tau,\ \tau<-1\}$ are isoparametric hypersurfaces.

We can also directly check that $\mathrm{O}^{+}(1,2)$ acts transitively
by isometries on $M_\tau$ by
$$
\psi(p,q)=(\psi p,\psi q),\ \ \psi\in \mathrm{O}^{+}(1,2).
$$
Thus, $\{M_\tau,\ \tau<-1\}$ is a family of homogeneous hypersurfaces.
\end{example}

At the end of this section, we give the following equivalent characterization
for the hypersurface $M_{\kappa,\tilde{\kappa}}^c$, which plays an important role
in the proof of Theorem \ref{thm:1.1}.

\begin{theorem}\label{thm:3.9}
Let $M$ be a hypersurface in $\mathbb{H}^{2} \times \mathbb{H}^{2}$ with constant $|C|<1$ .
Then, $J_1N+J_2N$ is a principal curvature vector field if and only if
$M$ is an open part of a hypersurface $M_{\kappa,\tilde{\kappa}}^c$ constructed by
\eqref{eqn:3.1} from two smooth curves $\gamma(r)$ and $\tilde{\gamma}(s)$ in $\mathbb{H}^{2}$,
where $c=\frac{1-C}{2}$.
\end{theorem}

\begin{proof}
To verify the ``only if" part, we assume that $J_1N+J_2N$
is a principal curvature vector field on $M$. Let
$$
E_1=\frac{J_1N+J_2N}{\sqrt{2(1+C)}},\ \ E_2=\frac{J_1N-J_2N}{\sqrt{2(1-C)}},\ \
E_3=\frac{V}{\sqrt{1-C^2}}.
$$
Then, $\{E_1, E_2, E_3\}$ forms an orthonormal frame on $M$, and it holds
\begin{equation}\label{eqn:3.7}
PE_1=E_1,\ PE_2=-E_2,\ PE_3=-CE_3+\sqrt{1-C^2}N,\ PN=CN+\sqrt{1-C^2}E_3.
\end{equation}
	
According to the fact that $C$ is constant and Remark \ref{rem:2.1},
it follows that $E_2$, $E_3$ are also principal curvature vector fields.
Thus, we can assume $A E_{1}=\lambda_1 E_{1}$, $A E_{2}=\lambda_2 E_{2}$, $A E_{3}=0$.
	
Now, by applying Remark \ref{rem:2.1}, $\bar{\nabla}J_i=0$ ($1\leq i\leq2$)
and $P=-J_1J_2=-J_2J_1$, we have
$$
\begin{aligned}
\nabla_{E_1} E_1&=(\bar{\nabla}_{E_1} E_1)^\top
=\tfrac{1}{\sqrt{2(1+C)}}\left(\bar{\nabla}_{E_1} (J_1N+J_2N)\right)^\top\\
&=\tfrac{1}{\sqrt{2(1+C)}}(-J_1AE_1-J_2AE_1)^\top=\tfrac{-\lambda_1}{\sqrt{2(1+C)}}(J_1E_1+J_2E_1)^\top\\
&=\tfrac{\lambda_1}{1+C}(PN+N)^\top=\tfrac{\lambda_1V}{1+C}=\lambda_1 \sqrt{\tfrac{1-C}{1+C}} E_{3},
\end{aligned}
$$
where $\cdot^\top$ means the tangential part. Then,
making use of \eqref{eqn:3.7}, by similar calculations we  obtain  the Levi-Civita connection $\nabla$
on $M$ given by
	
\begin{equation}\label{eqn:3.8}
\begin{aligned}
&\nabla_{E_{1}} E_{1}=\lambda_1 \sqrt{\tfrac{1-C}{1+C}} E_{3},\ \ \nabla_{E_{1}} E_{2}=0,
\ \ \nabla_{E_{1}} E_{3}=-\lambda_1 \sqrt{\tfrac{1-C}{1+C}} E_{1}, \\
&\nabla_{E_{2}} E_{1}=0,\ \ \nabla_{E_{2}} E_{2}=-\lambda_2 \sqrt{\tfrac{1+C}{1-C}} E_{3},
\ \ \nabla_{E_{2}} E_{3}=\lambda_2 \sqrt{\tfrac{1+C}{1-C}} E_{2}, \\
&\nabla_{E_{3}} E_{1}=\nabla_{E_{3}} E_{2}=\nabla_{E_{3}} E_{3}=0 .
\end{aligned}
\end{equation}
Combining the Codazzi equation \eqref{eqn:2.3},
we obtain
\begin{equation}\label{eqn:3.9}
E_3\lambda_1=-\tfrac{\sqrt{1-C^2}}{2}+\sqrt{\tfrac{1-C}{1+C}}\lambda_1^2,
\end{equation}
\begin{equation}\label{eqn:3.10}
E_3\lambda_2=\tfrac{\sqrt{1-C^2}}{2}-\sqrt{\tfrac{1+C}{1-C}}\lambda_2^2,
\end{equation}
\begin{equation}\label{eqn:3.11}
E_2\lambda_1=E_1\lambda_2=0.
\end{equation}
	
By solving the differential equations \eqref{eqn:3.9} and \eqref{eqn:3.10}  along
the integral curves of $E_3$, we have
\begin{equation}\label{eqn:3.12}
\begin{aligned}
&\lambda_1=\pm \sqrt{\tfrac{1+C}{2}} , {\rm or}\   \lambda_1=-\sqrt{\tfrac{1+C}{2}}\tanh\big(\sqrt{\tfrac{1-C}{2}}(t+h_1)\big), {\rm or}\ \lambda_1=-\sqrt{\tfrac{1+C}{2}}\coth\big(\sqrt{\tfrac{1-C}{2}}(t+h_1)\big),\\
&\lambda_2=\pm \sqrt{\tfrac{1-C}{2}}, {\rm or}\ \lambda_2=\sqrt{\tfrac{1-C}{2}}\tanh\big(\sqrt{\tfrac{1+C}{2}}(t+h_2)\big), {\rm or}\ \lambda_2=\sqrt{\tfrac{1-C}{2}}\coth\big(\sqrt{\tfrac{1+C}{2}}(t+h_2)\big),
\end{aligned}
\end{equation}
where the function $t$ satisfies $E_3t=1$ and the functions $h_1, h_2:M\rightarrow \mathbb{R}$
satisfy $E_3h_1=E_3h_2=0$.

\vskip1mm
{\bf Claim:} For any solution of $\{\lambda_1, \lambda_2\}$ in \eqref{eqn:3.12},
there are two appropriate nonzero functions $\rho_1,\rho_2$ on $M$,
such that the new frame
$$
\{X_1=\rho_1E_1,\ X_2=\rho_2E_2,\ X_3=E_3\}
$$
satisfies $[X_1,X_2]=[X_1,X_3]=[X_2,X_3]=0$.

\vskip1mm
{\it Proof of the Claim}.
We write the new frame
$$
\{X_1=\rho_1E_1,\ X_2=\rho_2E_2,\ X_3=E_3\},
$$
with two nonzero functions $\rho_1,\rho_2$  to be determined.
Specifically, by using $[E_i,E_j]=\nabla_{E_i} {E_j}-\nabla_{E_j} {E_i}$ and \eqref{eqn:3.8}, we get
$$
[E_1,E_2]=0,\ [E_1,E_3]=-\lambda_1\sqrt{\tfrac{1-C}{1+C}}E_1,\
[E_2,E_3]=\lambda_2\sqrt{\tfrac{1+C}{1-C}}E_2.
$$
Therefore
$$
\begin{aligned}
{\text [X_1,X_3]}&=-(\sqrt{\tfrac{1-C}{1+C}}\rho_1\lambda_1+E_3\rho_1)E_1.
\end{aligned}
$$
Then, $[X_1,X_3]=0$ is equivalent to
\begin{equation}\label{eqn:3.13}
E_3\rho_1=-\sqrt{\tfrac{1-C}{1+C}}\rho_1\lambda_1.
\end{equation}
	
Similarly, we get
$$
{\text [X_2,X_3]}=(\sqrt{\tfrac{1+C}{1-C}}\rho_2\lambda_2-E_3\rho_2)E_2,\ \
{\text [X_1,X_2]}=\rho_1(E_1\rho_2)E_2-\rho_2(E_2\rho_1)E_1.
$$
Then, $[X_2,X_3]=[X_1,X_2]=0$ are equivalent to
\begin{equation}\label{eqn:3.14}
E_3\rho_2=\sqrt{\tfrac{1+C}{1-C}}\rho_2\lambda_2,
\end{equation}
\begin{equation}\label{eqn:3.15}
E_1\rho_2=E_2\rho_1=0.
\end{equation}

In the following, we divide the discussions into three cases
according to whether $\lambda_1$ or $\lambda_2$ is constant.

\textbf{Case-i}: Both $\lambda_1$ and $\lambda_2$ are constant.

In this case, we first take a function $\rho=e^{\sqrt{\frac{1+C}{1-C}}\lambda_2 t}$,
where the function $t$ satisfies $E_3t=1$.
Now $\rho$ satisfies $E_3\rho=\sqrt{\tfrac{1+C}{1-C}}\rho\lambda_2$, which implies that $[\rho E_2,E_3]=0$.
Thus by canonical form of commuting vector fields, there exists a local coordinate $\{z_1,z_2,z_3\}$ on $M$
such that $\tfrac{\partial}{\partial z_2}=\rho E_2$ and $\tfrac{\partial}{\partial z_3}=E_3$.
Then, we have $E_2 z_3=\frac{1}{\rho}\tfrac{\partial z_3}{\partial z_2}=0$.

Next we take $\rho_1= e^{-\sqrt{\frac{1-C}{1+C}}\lambda_1 z_3}$,
which solves \eqref{eqn:3.13}. Thus it leads to $[X_1,X_3]=0$ and
$$
E_2\rho_1=-\sqrt{\tfrac{1-C}{1+C}}\lambda_1 e^{-\sqrt{\tfrac{1-C}{1+C}}\lambda_1 z_3}(E_2z_3)=0.
$$
Similarly we have a local coordinate $\{u_1,u_2,u_3\}$ on $M$ such that $\frac{\partial }{\partial u_1}=X_1=\rho_1E_1$ and $\tfrac{\partial}{\partial u_3}= E_3$.
So, we have $E_1 u_3=\frac{1}{\rho_1}\tfrac{\partial u_3}{\partial u_1}=0$.

Finally we set $\rho_2=e^{\sqrt{\frac{1+C}{1-C}}\lambda_2 u_3}$. Thus
\eqref{eqn:3.14} holds, and it implies $[X_2,X_3]=0$ and
$$
E_1\rho_2=\sqrt{\tfrac{1+C}{1-C}}\lambda_2 e^{\sqrt{\tfrac{1+C}{1-C}}\lambda_2 u_3}(E_1u_3)=0.
$$
It follows from $E_1\rho_2=E_2\rho_1=0$ that $[X_1,X_2]=0$.
Thus, we obtain the commuting frame $\{X_1,X_2,X_3\}$.

\textbf{Case-ii}: Only one of  $\lambda_1$ and $\lambda_2$ is constant.

Without loss of generality, we only consider the case that $\lambda_2=\sqrt{\tfrac{1-C}{2}}\tanh\big(\sqrt{\tfrac{1+C}{2}}(t+h_2)\big)$ is a function and
$\lambda_1$ is constant.
We first set $\rho_2=\cosh\big({\rm arctanh}(\sqrt{\tfrac{2}{1-C}}\lambda_2)\big)$.
By \eqref{eqn:3.10} and \eqref{eqn:3.11}, we see that $\rho_2$ satisfies \eqref{eqn:3.14}
and $E_1\rho_2=0$ holds, which implies $[X_2,X_3]=0$.
Similarly, by the canonical form of commuting vector fields,
there exists a local coordinate $\{w_1,w_2,w_3\}$ on $M$ such that
$\tfrac{\partial}{\partial w_2}=X_2=\rho_2 E_2$ and $\tfrac{\partial}{\partial w_3}= E_3$.
Then, we get $E_2 w_3=\frac{1}{\rho_2}\tfrac{\partial w_3}{\partial w_2}=0$.

Next set $\rho_1= e^{-\sqrt{\tfrac{1-C}{1+C}}\lambda_1 w_3}$,
which solves \eqref{eqn:3.13}. So, we have $[X_1,X_3]=0$ and
$$
E_2\rho_1=-\sqrt{\tfrac{1-C}{1+C}}\lambda_1 e^{-\sqrt{\tfrac{1-C}{1+C}}\lambda_1 w_3}(E_2w_3)=0.
$$
It follows from $E_1\rho_2=E_2\rho_1=0$ that $[X_1,X_2]=0$.
Thus we obtain the commuting frame $\{X_1,X_2,X_3\}$.
	
\textbf{Case-iii}: Neither $\lambda_1$ nor $\lambda_2$ are constant.

Due to the similarity of proofs, without loss of generality, here we only consider the case that
$\lambda_1=-\sqrt{\tfrac{1+C}{2}}\tanh\big(\sqrt{\tfrac{1-C}{2}}(t+h_1)\big)$, and
$\lambda_2=\sqrt{\tfrac{1-C}{2}}\tanh\big(\sqrt{\tfrac{1+C}{2}}(t+h_2)\big)$.
Let
\begin{equation}\label{eqn:3.16}
\rho_1=\cosh\big({\rm arctanh}(-\sqrt{\tfrac{2}{1+C}}\lambda_1)\big),\ \
\rho_2=\cosh\big({\rm arctanh}(\sqrt{\tfrac{2}{1-C}}\lambda_2)\big).
\end{equation}
By \eqref{eqn:3.9}--\eqref{eqn:3.10}, we can have that
$\rho_1$ and $\rho_2$ satisfy \eqref{eqn:3.13} and \eqref{eqn:3.14},
which implies that $[X_1,X_3]=[X_2,X_3]=0$.
By \eqref{eqn:3.11}, it follows that $E_1\rho_2=E_2\rho_1=0$, then we have $[X_1,X_2]=0$.
Thus, the frame $\{X_1, X_2, X_3\}$ satisfies $[X_1,X_2]=[X_1,X_3]=[X_2,X_3]=0$.
We have verified this Claim.

\vskip1mm

Based on the above claim, by Frobenius Theorem, we can identify $M$ with an open subset $\Omega$ of $\mathbb{R}^3$,
and express the hypersurface $M$ by an immersion
$$
\Phi: \Omega\subset\mathbb{R}^{3} \longrightarrow \mathbb{H}^{2} \times \mathbb{H}^{2},
\quad(t, r, s) \mapsto(p(t, r, s), q(t, r, s)),
$$
such that
$d\Phi(\tfrac{\partial}{\partial r})
=(\frac{\partial p}{\partial r},\frac{\partial q}{\partial r})=X_1$,
$d\Phi(\tfrac{\partial}{\partial s})
=(\frac{\partial p}{\partial s},\frac{\partial q}{\partial s})=X_2$
and
$d\Phi(\tfrac{\partial}{\partial t})
=(\frac{\partial p}{\partial t},\frac{\partial q}{\partial t})=X_3=E_3$.
	
By the definition of $P$, and using $P E_{1}=E_{1}$, $P E_{2}=-E_{2}$,
it follows that $dp, dq:T(\Omega)\rightarrow
T\mathbb{H}^2$ have the following properties:
\begin{equation}\label{eqn:3.17}
\left\{
\begin{aligned}
(dp(\tfrac{\partial}{\partial r}),0)&=\tfrac{1}{2}(d\Phi(\tfrac{\partial}{\partial r})
+Pd\Phi(\tfrac{\partial}{\partial r}))=d\Phi(\tfrac{\partial}{\partial r}),\\
(0,dq(\tfrac{\partial}{\partial r}))&=\tfrac{1}{2}(d\Phi(\tfrac{\partial}{\partial r})
-Pd\Phi(\tfrac{\partial}{\partial r}))=0,
\end{aligned}
\right.
\end{equation}
\vskip-3mm
\begin{equation}\label{eqn:3.18}
\left\{
\begin{aligned}
(dp(\tfrac{\partial}{\partial s}),0)&=\tfrac{1}{2}(d\Phi(\tfrac{\partial}{\partial s})
+Pd\Phi(\tfrac{\partial}{\partial s}))=0,\\
(0,dq(\tfrac{\partial}{\partial s}))&=\tfrac{1}{2}(d\Phi(\tfrac{\partial}{\partial s})
-Pd\Phi(\tfrac{\partial}{\partial s}))=d\Phi(\tfrac{\partial}{\partial s}).
\end{aligned}
\right.
\end{equation}
The first equation of \eqref{eqn:3.18} shows that $p$ depends only on $(t, r)$.
Similarly, from the second equation in \eqref{eqn:3.17} we derive that $q$ depends only on
$(t, s)$.
	
By $\langle PE_3,E_3\rangle=-C$, $AE_3=0$ and $\nabla_{E_3}{E_3}=0$, we know that
integral curves of $E_3=(v_{1}, v_{2})$ are geodesics of $\mathbb{H}^2\times\mathbb{H}^2$,
and $\|v_1\|=C^{-}$, $\|v_2\|=C^{+}$, where $C^{+}=\sqrt{\tfrac{1+C}{2}}$ and
$C^{-}=\sqrt{\tfrac{1-C}{2}}$. Now, we can assume that
$$
\begin{aligned}
&p(t, r)=\cosh \left(C^{-} t+c_1(r)\right) \gamma(r)+\sinh \left(C^{-} t+c_1(r)\right)
\left(m_{1}(r) \frac{d \gamma(r)}{d r}+m_{2}(r) N(r)\right),\\
&q(t, s)=\cosh \left(C^{+} t+c_2(s)\right) \tilde{\gamma}(s)+\sinh \left(C^{+} t+c_2(s)\right)
\left(n_{1}(s) \frac{d \tilde{\gamma}(s)}{d s}+n_{2}(s) \tilde{N}(s)\right),
\end{aligned}
$$
where $\gamma(r), \tilde{\gamma}(s)\in \mathbb{H}^{2}$, $\|\frac{d \gamma(r)}{d r}\|
=\|\frac{d \tilde{\gamma}(s)}{d s}\|=1$, and $N(r)$ and $\tilde{N}(s)$
are unit normal vector fields of $\gamma(r)$ and $\tilde{\gamma}(s)$ in $\mathbb{H}^{2}$,
respectively. Furthermore, $c_1(r), m_1(r), m_2(r)$ and $c_2(s), n_1(s), n_2(s)$
are functions on $\gamma(r)$ and $\tilde{\gamma}(s)$, respectively.
It also holds $m_{1}^{2}(r)+m_{2}^{2}(r)=n_{1}^{2}(s)+n_{2}^{2}(s)=1$.
	
Assume that $\kappa(r)$ and $\tilde{\kappa}(s)$ are curvatures of $\gamma(r)$
and $\tilde{\gamma}(s)$ in $\mathbb{H}^{2}$, respectively. Then we have
\begin{equation}\label{eqn:3.19}
\begin{aligned}
\frac{d^2\gamma(r)}{dr^2}&=\gamma(r)+\kappa(r)N(r), \ \
\frac{dN(r)}{dr}=-\kappa(r)\frac{d\gamma(r)}{dr},\\
\frac{d^2\tilde{\gamma}(s)}{ds^2}&=\tilde{\gamma}(s)+\tilde{\kappa}(s)\tilde{N}(s), \ \
\frac{d\tilde{N}(s)}{ds}=-\tilde{\kappa}(s)\frac{d\tilde{\gamma}(s)}{ds}.
\end{aligned}
\end{equation}
	
By using the fact $\langle E_1,E_3\rangle=\langle E_2,E_3\rangle=0$, we have
$$
\langle(\frac{\partial p}{\partial t},\frac{\partial q}{\partial t}),(\frac{\partial p}{\partial r},\frac{\partial q}{\partial r})\rangle=
\langle(\frac{\partial p}{\partial t},\frac{\partial q}{\partial t}),(\frac{\partial p}{\partial s},\frac{\partial q}{\partial s})\rangle=0,
$$
which implies that
\begin{equation}\label{eqn:3.20}
\frac{dc_1(r)}{dr}+m_1(r)=\frac{dc_2(s)}{ds}+n_1(s)=0.
\end{equation}
	
Now, we take a transformation to simplify the expression of $(p,q)$ as follows:
\begin{equation*}
\begin{aligned}
p(t,r)&=\cosh(\sqrt{\tfrac{1-C}{2}}t+c_1(r))\gamma(r)+\sinh(\sqrt{\tfrac{1-C}{2}}t+c_1(r))
(m_1(r)\frac{d\gamma(r)}{dr}+m_2(r)N(r))\\
&=\cosh(\sqrt{\tfrac{1-C}{2}}t)V_1(r)+\sinh(\sqrt{\tfrac{1-C}{2}}t)V_2(r),
\end{aligned}
\end{equation*}
where
$$
\begin{aligned}
V_1(r)&=\cosh(c_1(r))\gamma(r)+\sinh(c_1(r))(m_1(r)\frac{d\gamma(r)}{dr}+m_2(r)N(r)),\\
V_2(r)&=\sinh(c_1(r))\gamma(r)+\cosh(c_1(r))(m_1(r)\frac{d\gamma(r)}{dr}+m_2(r)N(r)).
\end{aligned}
$$
By using \eqref{eqn:3.20}, it can be checked that $\langle V_1(r),V_1(r)\rangle=-1$,
$\langle V_2(r),V_2(r)\rangle=1$ and
$\langle V_1(r),V_2(r)\rangle=\langle \frac{dV_1(r)}{dr},V_2(r)\rangle=0$,
i.e., $V_2(r)$ is the unit normal vector field of $V_1(r)\hookrightarrow \mathbb{H}^2$.
For $q(t,s)$, we can take a similar transformation, such that $q(t,s)$ can be written as
$$
q(t,s)=\cosh(\sqrt{\tfrac{1+C}{2}}t)W_1(s)+\sinh(\sqrt{\tfrac{1+C}{2}}t)W_2(s),
$$
where $W_2(s)$ is the unit normal vector field of $W_1(s)\hookrightarrow \mathbb{H}^2$.
	
Finally, by taking re-parameterizations to make $r$ and $s$ being arc length parameters,
we know that $M$ is an open part of a hypersurface $M_{\kappa,\tilde{\kappa}}^c$
described in Example \ref{exam:3.4}.
	
For the ``if" part, it is a direct consequence from verifying that
the hypersurface $M_{\kappa,\tilde{\kappa}}^c$ has constant product angle function $C=1-2c$, and
$J_1N+J_2N$ is a principal curvature vector field on $M_{\kappa,\tilde{\kappa}}^c$.
\end{proof}

\section{Proofs of Theorem \ref{thm:1.1} and Corollary \ref{cor:1.2}}\label{sect:4}

In this section, we study the hypersurfaces of $\mathbb{H}^{2} \times \mathbb{H}^{2}$
with constant principal curvatures and constant product angle function $C$.
When $C^2\neq 1$ holds on such hypersurfaces, by Remark \ref{rem:2.1}, we have $AV=0$.
Then, we give the following lemma, which will be used later.

\begin{lemma}\label{lemma:4.1}
Let $M$ be a hypersurface of $\mathbb{H}^{2} \times \mathbb{H}^{2}$ with constant principal
curvatures and constant $|C|<1$. Then, for any tangent vector fields $X,Y\in \{V\}^{\perp}$,
it holds
\begin{equation}\label{eqn:4.1}
-C\left\langle A^{2} X, Y\right\rangle+\langle A T A X, Y\rangle-\left\langle\nabla_{V} A X,
Y\right\rangle+\left\langle A \nabla_{V} X, Y\right\rangle=\tfrac{1}{2}\left(1-C^{2}\right)
\langle T X, Y\rangle.
\end{equation}
Here, $\{V\}^{\bot}$ denotes a distribution of $TM$ that is orthogonal to $V$.
Moreover, let $\{E_1,E_2,E_3=\tfrac{V}{\sqrt{1-C^2}}\}$ be a local orthonormal frame
satisfying $AE_i=\lambda_iE_i$, $1\leq i\leq3$, and $\lambda_3=0$.
Suppose that $\lambda_1\neq \lambda_2$, then we have
\begin{equation}\label{eqn:4.2}
\left\{
\begin{aligned}
\nabla_{E_1} E_1&=\tfrac{1}{\sqrt{1-C^2}}(P_{11}\lambda_{1}-C\lambda_{1})E_3,\ \ \ \
\nabla_{E_1} E_2=\tfrac{1}{\sqrt{1-C^2}}P_{12}\lambda_{1}E_3,\ \ \ \ \nabla_{E_3} E_3=0,\\
\nabla_{E_1} E_3&=\tfrac{1}{\sqrt{1-C^2}}[(C-P_{11})\lambda_{1}E_1-\lambda_{1}P_{12}E_2],
\ \ \ \ \qquad
\nabla_{E_2} E_1=\tfrac{1}{\sqrt{1-C^2}}P_{12}\lambda_{2}E_3,\\
\nabla_{E_2} E_2&=\tfrac{1}{\sqrt{1-C^2}}(P_{22}\lambda_{2}-C\lambda_{2})E_3,\ \
\nabla_{E_2} E_3=\tfrac{1}{\sqrt{1-C^2}}[(C-P_{22})\lambda_{2}E_2-\lambda_{2}P_{12}E_1],\\
\nabla_{E_3} E_1&=\tfrac{P_{12}}{\lambda_{1}-\lambda_{2}}(\tfrac{\lambda_{1}\lambda_{2}}
{\sqrt{1-C^2}}-\tfrac{\sqrt{1-C^2}}{2})E_2,\ \ \ \
\nabla_{E_3} E_2=-\tfrac{P_{12}}{\lambda_{1}-\lambda_{2}}(\tfrac{\lambda_{1}\lambda_{2}}
{\sqrt{1-C^2}}-\tfrac{\sqrt{1-C^2}}{2})E_1,
\end{aligned}\right.
\end{equation}
where $P_{ij}=\langle PE_i,E_j\rangle$, $1\leq i,j\leq2$.
\end{lemma}
\begin{proof}

We assume that $M$ is a hypersurface of $\mathbb{H}^{2}\times\mathbb{H}^{2}$
with constant principal curvatures and constant $|C|<1$. Then, by Codazzi equation
\eqref{eqn:2.3}, we have
$$
\left(\nabla_{X} A\right) V-\left(\nabla_{V} A\right) X=\tfrac{1}{2}\left(1-C^{2}\right) T X,
\quad  \forall X \in\{V\}^{\perp}.
$$
On the other hand, by using Lemma \ref{lemma:2.1} and $AV=0$, we get
$$
\begin{aligned}
\left(\nabla_{X} A\right) V-\left(\nabla_{V} A\right) X &
=-A \nabla_{X} V-\nabla_{V} A X+A \nabla_{V} X \\
&=-C A^{2} X+A T A X-\nabla_{V} A X+A \nabla_{V} X.
\end{aligned}
$$
From above two equations, it follows
\begin{equation}\label{eqn:4.3}
-C A^{2} X+A T A X-\nabla_{V} A X+A \nabla_{V} X=\tfrac{1}{2}\left(1-C^{2}\right) T X,
\quad  \forall X \in\{V\}^{\perp}.
\end{equation}
By taking the inner product of \eqref{eqn:4.3} with $Y \in\{V\}^{\perp}$,
we can obtain \eqref{eqn:4.1}.

In the following, we take a local orthonormal frame field
$\{E_1, E_2, E_3=\tfrac{V}{\sqrt{1-C^2}}\}$ such that
$AE_i=\lambda_iE_i$, $1\leq i\leq3$, and $\lambda_3=0$.
Suppose that $\lambda_1\neq \lambda_2$.

Then, by applying Codazzi equation \eqref{eqn:2.3}, we have
$(\nabla_{E_1}A)E_2=(\nabla_{E_2}A)E_1$. It follows
\begin{align*}
0=&\langle (\nabla_{E_i}A)E_j-(\nabla_{E_j}A)E_i,E_i\rangle\\
=&(\lambda_{j}-\lambda_{i})\langle \nabla_{E_i}E_j,E_i\rangle
-(\lambda_{i}-\lambda_{i})\langle\nabla_{E_j}E_i,E_i\rangle\\
=&(\lambda_{i}-\lambda_{j})\langle\nabla_{E_i}E_i,E_j\rangle, \ 1\leq i\neq j\leq2.
\end{align*}
From $\lambda_1\neq\lambda_2$, we deduce $\langle \nabla_{E_i}E_i,E_j\rangle=0$ for
$1\leq i\neq j\leq2$. Therefore, by Lemma \ref{lemma:2.1}, we have
\begin{align*}
\nabla_{E_{1}}E_1&=\langle\nabla_{E_{1}}E_1,E_3\rangle E_3
=-\tfrac{1}{\sqrt{1-C^2}}\langle\nabla_{E_{1}}V,E_1\rangle E_3\\
&=-\tfrac{1}{\sqrt{1-C^2}}\langle CAE_1-TAE_1,E_1\rangle E_3
=\tfrac{1}{\sqrt{1-C^2}}(P_{11}\lambda_{1}-C\lambda_{1})E_3.
\end{align*}

Taking $(X, Y)=\left(E_{1}, E_{2}\right)$ into \eqref{eqn:4.1}, we have
$$
\lambda_{1}\lambda_{2}P_{12}-(\lambda_{1}-\lambda_{2})\langle \nabla_{V}E_1,E_2\rangle
=\tfrac{1-C^2}{2}P_{12},
$$
which implies that $\langle \nabla_{E_3}E_1,E_2\rangle=\tfrac{P_{12}}{\lambda_{1}-\lambda_{2}}
(\tfrac{\lambda_{1}\lambda_{2}}{\sqrt{1-C^2}}-\tfrac{\sqrt{1-C^2}}{2})$.
By Lemma \ref{lemma:2.1}, we have
$$
\langle\nabla_{E_{3}}E_1,E_3\rangle=-\tfrac{1}{\sqrt{1-C^2}}\langle\nabla_{E_{3}}V,E_1\rangle
=-\tfrac{1}{\sqrt{1-C^2}}\langle CAE_3-TAE_3,E_1\rangle=0.
$$
It follows that
$\nabla_{E_3}E_1=\tfrac{P_{12}}{\lambda_{1}-\lambda_{2}}(\tfrac{\lambda_{1}\lambda_{2}}
{\sqrt{1-C^2}}-\tfrac{\sqrt{1-C^2}}{2})E_2$. The rest connections can also be obtained
by totally similar calculations.
\end{proof}

\noindent{\bf Proof of Theorem \ref{thm:1.1}}.

We assume that $M$ is a hypersurface of $\mathbb{H}^{2} \times \mathbb{H}^{2}$
with constant principal curvatures and constant product angle function $C$.
When $C=\pm1$ on $M$, then by Lemmas \ref{lemma:3.2}--\ref{lemma:3.3},
we know that $M$ is an open part of $M_\Gamma$, where $\Gamma$ is a curve of $\mathbb{H}^2$
with constant curvature. More specifically, if $M$ is umbilical,
then $\Gamma$ is a geodesic of $\mathbb{H}^2$.
If $M$ has two distinct principal curvatures, then $\Gamma$ is a curve of $\mathbb{H}^2$
with nonzero constant curvature. In the following, we always assume $C\neq \pm1$ on $M$.

By Remark \ref{rem:2.1}, we have that $V$ is a principal curvature vector field satisfying $AV=0$.
Now, we divide the discussions into three cases depending on the number $g$
of distinct principal curvatures.
	
\textbf{Case (1)}: $g=1$.
	
In this case, by $AV=0$,  $M$ is totally geodesic.
Taking $(X, Y)=(V, \frac{J_1N+J_2N}{\sqrt{2(1+C)}})$
in Codazzi equation \eqref{eqn:2.3}, we can get $C=\pm1$, which is a contradiction.
Thus, this case does not occur.
	
\textbf{Case (2)}: $g=2$.

In this case, since $0$ is a principal curvature of $M$, there are two subcases:
$0$ is a principal curvature with multiplicity $1$; or $0$ is a principal curvature
with multiplicity $2$. We will discuss these two subcases, respectively.

\textbf{Case (2)-i}: $0$ is a principal curvature with multiplicity $1$.

In this subcase, we have $AX=\lambda X$ for any $X\in\{V\}^\bot$, where $\lambda$ is
a nonzero constant on $M$. Then, from the fact that $J_1N+J_2N\in \{V\}^\bot$,
it follows that $J_1N+J_2N$ is a principal curvature vector
field of $M$. Thus, by Theorem \ref{thm:3.9}, we know that $M$ is an open part of
hypersurface $M_{\kappa,\tilde{\kappa}}^c$ constructed by
\eqref{eqn:3.1}, where $c=\frac{1-C}{2}$, and $N(r)$ and $\tilde{N}(s)$ are unit normal vector
fields of curves $\gamma(r)$ and $\tilde{\gamma}(s)$ in $\mathbb{H}^{2}$, such that
these two curves' curvatures are $\kappa(r)$ and $\tilde{\kappa}(s)$, respectively.

Notice that, $M_{1,1}^{c}$ and $M_{1,-1}^{c}$ $(0<c<1)$ are the only hypersurfaces
in $M_{\kappa,\tilde{\kappa}}^c$ with constant principal curvatures. From
the computations of principal curvatures of hypersurfaces $M_{1,1}^{c}$ and $M_{1,-1}^{c}$,
we know that $M$ is an open part of $M_{1,-1}^{1/2}$.

\textbf{Case (2)-ii}: $0$ is a principal curvature with multiplicity $2$.

In this subcase, we can take a local orthonormal frame field
$\{E_1,E_2,E_3=\tfrac{V}{\sqrt{1-C^2}}\}$, which satisfies
$$
AE_1=0,\ AE_2=\lambda_2E_2,\ AE_3=0, \ \lambda_2\neq0.
$$

Let $P_{ij}=\langle PE_i,E_j\rangle$, $1\leq i,j\leq2$.
Taking $(X, Y)=(E_{1}, E_{1})$ and $(X, Y)=(E_2,E_2)$ into \eqref{eqn:4.1}, we have
\begin{equation}\label{eqn:4.4}
\tfrac{1}{2}(1-C^2)P_{11}=0,
\end{equation}
\begin{equation}\label{eqn:4.5}
-C\lambda_{2}^2+\lambda_{2}^2P_{22}=\tfrac{1}{2}(1-C^2)P_{22}.
\end{equation}

By $C\neq \pm1$, $P_{11}+P_{22}=0$ and $P^2_{11}+P^2_{12}=1$, equation \eqref{eqn:4.4}
implies that $P_{11}=P_{22}=0$ and $P_{12}=\pm1$. Up to a sign of $E_1$, we can assume $P_{12}=1$.
Then, by $P_{22}=0$ and $\lambda_2\neq0$, \eqref{eqn:4.5} follows $C=0$.
Now, by Lemma \ref{lemma:4.1}, the information of connections with respect to
$\{E_i\}_{i=1}^3$ is given by
\begin{align*}
&\nabla_{E_1}E_1=\nabla_{E_1}E_2=\nabla_{E_1}E_3=0,\\
&\nabla_{E_2}E_1=\lambda_{2}E_3,\ \quad\nabla_{E_2}E_2=0 , \ \quad\nabla_{E_2}E_3=-\lambda_{2}E_1,\\
&\nabla_{E_3}E_1=\tfrac{1}{2\lambda_{2}} E_2,
\quad \nabla_{E_3}E_2=-\tfrac{1}{2\lambda_{2}} E_1,\quad \nabla_{E_3}E_3=0.
\end{align*}

Next, by definition of the curvature tensor $R$, we have
\begin{align*}
R(E_1,E_2)E_1&=\nabla_{E_1}\nabla_{E_2}E_1-\nabla_{E_2}\nabla_{E_1}E_1-\nabla_{[E_1,E_2]}E_1\\
&=\lambda_{2}\nabla_{E_3}E_1=\tfrac{1}{2}E_2.
\end{align*}
On the other hand, by Gauss equation \eqref{eqn:2.2}, it follows
\begin{align*}
R(E_1,E_2)E_1=&-\tfrac{1}{2}[\langle E_2,E_1\rangle E_1-\langle E_1,E_1\rangle E_2
+\langle TE_2,E_1\rangle TE_1-\langle TE_1,E_1\rangle TE_2]\\
&+\langle AE_2,E_1\rangle AE_1-\langle AE_1,E_1\rangle AE_2=0,
\end{align*}
which contradicts $R(E_1,E_2)E_1=\tfrac{1}{2}E_2$. Thus, this subcase does not occur.

\textbf{Case (3)}: $g=3$.

In this case, we can take a local orthonormal frame $\{E_1,E_2,E_3=\tfrac{V}{\sqrt{1-C^2}}\}$,
which satisfies
$$
AE_1=\lambda_1E_1,\ AE_2=\lambda_2E_2,\ AE_3=0, \ \lambda_1\neq0,\ \lambda_2\neq0, \
\lambda_1\neq\lambda_2.
$$
Let $P_{ij}=\langle PE_i,E_j\rangle$, $1\leq i,j\leq2$. Then, \eqref{eqn:4.2} holds on $M$,
and we have
\begin{equation}\label{eqn:4.6}
\begin{aligned}
PE_1&=P_{11}E_1+P_{12}E_2,\ \ \ \ \ \qquad PE_2=P_{12}E_1+P_{22}E_2,\\
PE_3&=-CE_3+\sqrt{1-C^2}N,\ \ PN=\sqrt{1-C^2}E_3+CN.
\end{aligned}
\end{equation}

Taking $(X, Y)=(E_{1}, E_{1})$ and $(X, Y)=(E_{2}, E_{2})$ in \eqref{eqn:4.1}, we have
\begin{equation}\label{eqn:4.7}
(-1+C^2+2\lambda_{1}^2)P_{11}=2C\lambda_{1}^2,
\end{equation}
\begin{equation}\label{eqn:4.8}
(-1+C^2+2\lambda_{2}^2)P_{22}=2C\lambda_{2}^2.
\end{equation}

In the following, we divide the discussions into two subcases depending on $C=0$ or $C\neq0$.

\textbf{Case (3)-i}: $C\neq 0$.

In this subcase, by \eqref{eqn:4.7}, $P_{11}+P_{22}=0$ and $C\neq 0$, we have
$P_{11}=-P_{22}=\tfrac{2C\lambda_{1}^2 }{-1+C^2+2\lambda_{1}^2}$.
Substituting $P_{22}=\tfrac{-2C\lambda_{1}^2 }{-1+C^2+2\lambda_{1}^2}$
into equation \eqref{eqn:4.8}, with the use of $C\neq 0$, we get
$$
C^2(\lambda_{1}^2+\lambda_{2}^2)=\lambda_{1}^2+\lambda_{2}^2-4\lambda_{1}^2\lambda_{2}^2,\ \
\lambda_{1}^2+\lambda_{2}^2-4\lambda_{1}^2\lambda_{2}^2\neq 0.
$$
Thus,
$C^2=\tfrac{\lambda_{1}^2+\lambda_{2}^2-4\lambda_{1}^2\lambda_{2}^2}
{\lambda_{1}^2+\lambda_{2}^2}$.

Now, using \eqref{eqn:4.2} and \eqref{eqn:4.6} and calculating
the right hand side of $0=\langle(\bar{\nabla}_{E_3}P)E_2,E_2\rangle$,
we can have $P_{12}^2(-1+C^2+2\lambda_1\lambda_2)=0$. If $P_{12}\neq0$, it follows
$C^2=1-2\lambda_1\lambda_2$. Comparing it with
$C^2=\tfrac{\lambda_{1}^2+\lambda_{2}^2-4\lambda_{1}^2\lambda_{2}^2}
{\lambda_{1}^2+\lambda_{2}^2}$, we get $\tfrac{-2\lambda_1(\lambda_1-\lambda_2)^2\lambda_2}
{\lambda_1^2+\lambda_2^2}=0$, which is a contradiction.
Thus, it holds $P_{12}=0$ on $M$, which implies that $PE_1=\pm E_1$ and $PE_2=\mp E_2$.

Due to $P(J_1N+J_2N)=J_1N+J_2N$, it follows that $J_1N+J_2N$ is parallel to $E_1$
or $E_2$, which means that $J_1N+J_2N$ is a principal curvature vector field of $M$.
Then, by Theorem \ref{thm:3.9}, we know that $M$ is an open part of
hypersurface $M_{\kappa,\tilde{\kappa}}^c$ constructed by
\eqref{eqn:3.1}, where $c=\frac{1-C}{2}$, and $N(r)$ and $\tilde{N}(s)$ are unit normal vector
fields of curves $\gamma(r)$ and $\tilde{\gamma}(s)$ in $\mathbb{H}^{2}$, such that
the curvatures of these two curves  are $\kappa(r)$ and $\tilde{\kappa}(s)$, respectively.

From the fact that $M_{1,1}^{c}$ and $M_{1,-1}^{c}$ $(0<c<1)$ are the only hypersurfaces
in $M_{\kappa,\tilde{\kappa}}^c$ with constant principal curvatures, by applying
the computations of principal curvatures of hypersurfaces $M_{1,1}^{c}$ and $M_{1,-1}^{c}$,
we know that $M$ is either an open part of $M_{1,-1}^{c}$ for some
$c\in(0,\tfrac{1}{2})\cup(\tfrac{1}{2},1)$, or $M$ is an open part of
$M_{1,1}^{c}$ for some $c\in(0,\tfrac{1}{2})\cup(\tfrac{1}{2},1)$.

\textbf{Case (3)-ii}: $C=0$.

In this subcase, taking $(X, Y)=(E_{1}, E_{1})$ and $(X, Y)=(E_2,E_2)$ in \eqref{eqn:4.1},
we get
\begin{equation}\label{eqn:4.9}
P_{11}(\lambda_{1}^2-\tfrac{1}{2})=P_{22}(\lambda_{2}^2-\tfrac{1}{2})=0.
\end{equation}

In the following, we further divide the discussions into two subcases
depending on the value of $\lambda_1$ and $\lambda_2$.

\textbf{Case (3)-ii-i}: $\lambda^2_{1}\neq\frac{1}{2}$ or $\lambda^2_{2}\neq\frac{1}{2}$.

\textbf{Case (3)-ii-ii}: $\lambda^2_{1}=\lambda^2_{2}=\frac{1}{2}$.

In \textbf{Case (3)-ii-i}, by \eqref{eqn:4.9} and $P_{11}+P_{22}=0$,
we have $P_{11}=P_{22}=0$ and $P_{12}=\pm1$. Up to a sign of $E_1$, we can assume $P_{12}=1$.
Now, by Gauss equation \eqref{eqn:2.2}, with the use of \eqref{eqn:4.6}, it holds
$$
R(E_1,E_2)E_1=-\lambda_1\lambda_2E_2.
$$
On the other hand, by definition of the curvature tensor $R$, with the use of \eqref{eqn:4.2},
we have
\begin{align*}
R(E_1,E_2)E_1&=\nabla_{E_1}\nabla_{E_2}E_1-\nabla_{E_2}\nabla_{E_1}E_1-\nabla_{[E_1,E_2]}E_1\\
&=(\tfrac{1}{2}-2\lambda_1\lambda_2)E_2.
\end{align*}
It follows $\lambda_{1}\lambda_{2}=\tfrac{1}{2}$. Up to a sign of $N$, we assume $\lambda_1>0$
and $\lambda_2>0$. From $\langle J_1N,J_2N\rangle=C=0$, we assume $E_1=aJ_1N+bJ_2N$,
where $a, b$ are functions, and it holds $a^2+b^2=1$. By using $P_{11}=0$, $P_{12}=1$,
$PJ_1N=J_2N$ and $PJ_2N=J_1N$, we have $E_2=PE_1=bJ_1N+aJ_2N$.
Then, by $\langle E_1,E_2\rangle=0$, we get $ab=0$. Thus, without loss of generality,
by switching the roles of $E_1$ and $E_2$, we can assume $E_1=J_1N$, $E_2=J_2N$ and
$$
AV=0,\quad A(J_1N)=\lambda_1J_1N, \quad  A(J_2N)=\lambda_2J_2N.
$$

As mentioned in Subsection \ref{sect:2.1}, if $\mathcal{F}:\mathbb{H}^2\to \mathbb{H}^2$
is an anti-holomorphic isometry, then
$$
\tilde{\mathcal{F}}=(\text{Id},\mathcal{F}) : \mathbb{H}^2\times \mathbb{H}^2\to
\mathbb{H}^2\times \mathbb{H}^2
$$
is a holomorphic isometry from $(\mathbb{H}^2\times \mathbb{H}^2, \langle\cdot,\cdot\rangle , J_1)$
onto $(\mathbb{H}^2\times \mathbb{H}^2, \langle\cdot,\cdot\rangle , J_2)$. It follows that
the differential of the isometry $\tilde{\mathcal{F}}$ satisfies
\begin{equation}\label{eqn:4.10}
J_2\circ d\tilde{\mathcal{F}}=d\tilde{\mathcal{F}}\circ J_1,
\ \ P\circ d\tilde{\mathcal{F}}=d\tilde{\mathcal{F}}\circ P.
\end{equation}

By using the isometry $\tilde{\mathcal{F}}$, we obviously get another hypersurface
$\tilde{\mathcal{F}}(M)$ of $\mathbb{H}^{2} \times \mathbb{H}^{2}$
with constant principal curvatures and constant product angle function $C=0$.
Noticing that $\tilde{N}=d\tilde{\mathcal{F}}(N)$ is the unit normal vector field of
$\tilde{\mathcal{F}}(M)$. By using \eqref{eqn:4.10}, we have
$$
\begin{aligned}
&J_2\tilde{N}=J_2d\tilde{\mathcal{F}}(N)=d\tilde{\mathcal{F}}(J_1N),\\
&J_1\tilde{N}=PJ_2\tilde{N}=PJ_2d\tilde{\mathcal{F}}(N)=Pd\tilde{\mathcal{F}}(J_1N)
=d\tilde{\mathcal{F}}(PJ_1N)=d\tilde{\mathcal{F}}(J_2N).
\end{aligned}
$$
From the fact that $\tilde{\mathcal{F}}$ is an isometry of $\mathbb{H}^2\times \mathbb{H}^2$,
we have
$$
\begin{aligned}
&\tilde{A}(J_1\tilde{N})=-\bar{\nabla}_{J_1\tilde{N}}{\tilde{N}}
=-d\tilde{\mathcal{F}}(\bar{\nabla}_{J_2N}N)
=d\tilde{\mathcal{F}}(A(J_2N))=\lambda_2J_1\tilde{N}=\tilde{\lambda}_1J_1\tilde{N},\\
&\tilde{A}(J_2\tilde{N})=-\bar{\nabla}_{J_2\tilde{N}}{\tilde{N}}
=-d\tilde{\mathcal{F}}(\bar{\nabla}_{J_1N}N)
=d\tilde{\mathcal{F}}(A(J_1N))=\lambda_1J_2\tilde{N}=\tilde{\lambda}_2J_2\tilde{N},
\end{aligned}
$$
where $\tilde{A}$ is the shape operator of $\tilde{\mathcal{F}}(M)$.
If $\lambda_1<\lambda_2$ on $M$, then it holds $\tilde{\lambda}_1>\tilde{\lambda}_2$
on $\tilde{\mathcal{F}}(M)$. Thus, without loss of generality,
we can assume that $\lambda_1>\lambda_2$ and
$$
AV=0,\quad A(J_1N)=\tfrac{1}{\sqrt{2}}\sqrt{\tfrac{\tau-1}{\tau+1}}J_1N,
\quad  A(J_2N)=\tfrac{1}{\sqrt{2}}\sqrt{\tfrac{\tau+1}{\tau-1}}J_2N,
$$
where $\tau<-1$. Since $C=0$, we have $N=(N_1,N_2)$ with $\|N_1\|^2=\|N_2\|^2=\frac{1}{2}$.
Now, the parallel hypersurfaces to $M$ are given by
$\Phi_l: M\to \mathbb{H}^2\times\mathbb{H}^2$, where
$$
\Phi_l=(\operatorname{exp}_p(lN_1),\operatorname{exp}_q(lN_2))=\cosh(\tfrac{l}{\sqrt{2}})(p,q)
+\sqrt{2}\sinh(\tfrac{l}{\sqrt{2}})N_{(p,q)},\ \ (p,q)\in M.
$$
Then, by direct calculations, we have
\begin{align*}
(\Phi_l)_{*}(V)&=\cosh(\tfrac{l}{\sqrt{2}})V+\tfrac{1}{\sqrt{2}}\sinh(\tfrac{l}{\sqrt{2}})(p,-q),\\
(\Phi_l)_{*}(J_1N)&=[\cosh(\tfrac{l}{\sqrt{2}})
-\sqrt{\tfrac{\tau-1}{\tau+1}}\sinh(\tfrac{l}{\sqrt{2}})]J_1N,\\
(\Phi_l)_{*}(J_2N)&=[\cosh(\tfrac{l}{\sqrt{2}})
-\sqrt{\tfrac{\tau+1}{\tau-1}}\sinh(\tfrac{l}{\sqrt{2}})]J_2N.
\end{align*}
Thus, when $\coth(\frac{l}{\sqrt{2}})=\sqrt{\tfrac{\tau-1}{\tau+1}}$,
$\Phi_l(M)$ is a focal surface of $M$.
From $\coth(\frac{l}{\sqrt{2}})=\sqrt{\tfrac{\tau-1}{\tau+1}}$, we have
$\cosh(\tfrac{l}{\sqrt{2}})=\sqrt{\tfrac{1-\tau}{2}}$ and
$\sinh(\tfrac{l}{\sqrt{2}})=\sqrt{-\tfrac{\tau+1}{2}}$,
which implies that $\cosh(\sqrt{2}l)=-\tau$ and $l=\tfrac{1}{\sqrt{2}}{\rm arccosh}(-\tau)$.

Now, let $l=\tfrac{1}{\sqrt{2}}{\rm arccosh}(-\tau)$ and $\Sigma:=\Phi_l(M)$
be a focal surface of $M$, then $\Sigma$ is given by
$$
\Sigma=\sqrt{\tfrac{1-\tau}{2}}(p,q)+\sqrt{-\tau-1}N_{(p,q)}, \ \ (p,q)\in M.
$$
By $(\Phi_{l})_{*}(J_1 N)=0$, it follows that $\{(\Phi_l)_{*}(V), J_{2} N\}$
is an orthonormal frame of the tangent bundle of $\Sigma$, and it holds
$$
(\Phi_{l})_{*}(V)=\sqrt{\tfrac{1-\tau}{2}} V+\tfrac{\sqrt{-1-\tau}}{2}(p,-q),\ \
(\Phi_{l})_{*}\left(J_{2} N\right)=\sqrt{\tfrac{2}{1-\tau}}J_{2} N.
$$
So, for $(p,q)\in M$,
$\left\{J_{1} N, \tfrac{\sqrt{-1-\tau}}{2}(p,q)+\sqrt{\tfrac{1-\tau}{2}} N\right\}$
is an orthonormal frame on the normal bundle of $\Sigma$. Now, by direct calculations,
it can be checked that the corresponding Weingarten endomorphisms associated to these
two unit normal vector fields vanish, which means that $\Sigma$ is totally geodesic.
Next, by definition of the complex structure $J_1$ and $P=-J_1J_2=-J_2J_1$, we have
$$
\begin{aligned}
J_1(\Phi_{l})_{*}(V)&=((\sqrt{\tfrac{1-\tau}{2}}p+\sqrt{-\tau-1}N_1)
\boxtimes(\sqrt{\tfrac{1-\tau}{2}}N_1+\tfrac{\sqrt{-1-\tau}}{2}p),\\
&\ \ \ \ (\sqrt{\tfrac{1-\tau}{2}}q+\sqrt{-\tau-1}N_2)
\boxtimes(-\sqrt{\tfrac{1-\tau}{2}}N_2-\tfrac{\sqrt{-1-\tau}}{2}q))\\
&=(p\boxtimes N_1,-q\boxtimes N_2)=J_2N,
\end{aligned}
$$
$$
\langle J_2(\Phi_{l})_{*}(V),J_2N\rangle=\langle PJ_1(\Phi_{l})_{*}(V),J_2N\rangle
=\langle PJ_2N,J_2N\rangle=0.
$$
So, $\Sigma$ is an almost complex surface with respect to
the complex structure $J_1$, and a Lagrangian surface with respect to the
complex structure $J_2$. From Theorem 4.2 of \cite{GVWX}, we have that $\Sigma$ is congruent to
an open part of the diagonal surface $\left\{(p, p) \in \mathbb{H}^{2}\times\mathbb{H}^{2}
\mid p \in \mathbb{H}^{2}\right\}$, and $M$ is an open part of the tube of
radius $l=\tfrac{1}{\sqrt{2}}{\rm arccosh} (-\tau)$ over the diagonal surface. According to
Example \ref{exam:3.8}, we know that $M$ is an open part of $M_\tau$ for some $\tau<-1$.

In \textbf{Case (3)-ii-ii}, up to a sign of $N$, without loss of generality, we may assume
$\lambda_{1}=-{\tfrac{1}{\sqrt 2}},\lambda_{2}={\tfrac{1}{\sqrt 2}}$.
Now, by Gauss equation \eqref{eqn:2.2}, with the use of \eqref{eqn:4.6}, we have
\begin{align*}
\langle R(E_1,E_2)E_1,E_2\rangle=\tfrac{1}{2}(2-P_{12}^2+P_{11}P_{22}).
\end{align*}
On the other hand, by definition of the curvature tensor $R$, with the use of \eqref{eqn:4.2}
and \eqref{eqn:4.6}, we have
\begin{align*}
\langle R(E_1,E_2)E_1,E_2\rangle=\tfrac{1}{2}(3P_{12}^2-P_{11}P_{22}).
\end{align*}
It follows
\begin{equation}\label{eqn:4.11}
-1+2P_{12}^2-P_{11}P_{22}=0.
\end{equation}
By using $P_{11}^2+P_{12}^2=1$ and $P_{11}+P_{22}=0$, equation \eqref{eqn:4.11}
implies that $P_{12}=0$, and then $PE_1=\pm E_1$ and $PE_2=\mp E_2$.

Due to $P(J_1N+J_2N)=J_1N+J_2N$, it follows that $J_1N+J_2N$
is a principal curvature vector field of $M$. Then, by Theorem \ref{thm:3.9},
we know that $M$ is an open part of hypersurface $M_{\kappa,\tilde{\kappa}}^c$ constructed by
\eqref{eqn:3.1}, where $c=\frac{1-C}{2}$, and $N(r)$ and $\tilde{N}(s)$ are unit normal vector
fields of curves $\gamma(r)$ and $\tilde{\gamma}(s)$ in $\mathbb{H}^{2}$, such that
these two curves' curvatures are $\kappa(r)$ and $\tilde{\kappa}(s)$, respectively.

From the fact that $M_{1,1}^{c}$ and $M_{1,-1}^{c}$ $(0<c<1)$ are the only hypersurfaces
in $M_{\kappa,\tilde{\kappa}}^c$ with constant principal curvatures, by applying
the computations of principal curvatures of hypersurfaces $M_{1,1}^{c}$ and $M_{1,-1}^{c}$,
we know that $M$ is an open part of $M_{1,1}^{1/2}$.

In conclusion, we have completed the proof of Theorem \ref{thm:1.1}.
\qed

\vskip3mm

\noindent{\bf Proof of Corollary \ref{cor:1.2}}.

Let $N$ be a unit normal vector field of $M$. We fix a point $z_{0} \in M$.
From the assumption that $M$ is homogeneous, for any $z \in M$,
there exists an isometry $\mathbb{F}$ of $\mathbb{H}^{2} \times \mathbb{H}^{2}$
such that $\mathbb{F}\left(M\right)=M$ and $\mathbb{F}\left(z_{0}\right)=z$.
It follows that $N_{z}=\pm d \mathbb{F}_{z_{0}}\left(N_{z_{0}}\right)$
and $d \mathbb{F}_{z_{0}} \circ P=\pm P \circ d \mathbb{F}_{z_{0}}$, which implies that
$C(z)=\pm C\left(z_{0}\right)$. From the fact that $M$ is connected,
the product angle function $C$ is constant on $M$. On the other hand,
isometry $\mathbb{F}$ preserves the second fundamental form,
then we can see that the principal curvatures of $M$ are also constant.
From the fact that a homogeneous hypersurface is always complete,
then the result follows from Theorem \ref{thm:1.1}.
\qed

\section{Proof of Theorem \ref{thm:1.3}}\label{sect:5}

In this section, by employing Jacobi field computation
then using an argument as in \cite{Ur}, we classify isoparametric hypersurfaces of
$\mathbb{H}^{2} \times \mathbb{H}^{2}$. We first discuss some basic properties
about nearby parallel hypersurfaces of a general hypersurface.

Let $M$ be a hypersurface of $\mathbb{H}^{2} \times \mathbb{H}^{2}$ with unit normal
vector field $N$, and $\mathcal{I}: M\rightarrow \mathbb{H}^{2} \times \mathbb{H}^{2}$
denote the inclusion map. We consider the open subset of $M$ defined by
$\Omega=\{z \in M \mid C^{2}(z)<1\}$.
If $\Omega$ is empty, then $C^{2}=1$ holds on $M$. From Lemma \ref{lemma:3.3},
$M$ is an open part of $M_\Gamma$ for some curve $\Gamma$ in $\mathbb{H}^2$.
The parallel hypersurface of $M_\Gamma$ at distance $l$ is $\tilde{\Gamma}\times \mathbb{H}^2$,
where $\tilde{\Gamma}$ is a parallel curve of $\Gamma$ at distance $l$
in $\mathbb{H}^2$.

Now, we suppose that $\Omega$ is not empty, and consider the nearby parallel
hypersurfaces to any given connected open subset of $\Omega$, still denoted by $\Omega$.
If $N=\left(N_{1}, N_{2}\right)$ is a unit normal vector field to $M$, then the nearby parallel
hypersurfaces $\Phi_{l}: \Omega \rightarrow \mathbb{H}^{2} \times \mathbb{H}^{2}$
are given by
$$
\Phi_l(p,q)=\left(\exp _{p} (l N_{1}), \exp _{q} (l N_{2})\right), \quad (p,q)\in\Omega,
\quad l \in(-\epsilon, \epsilon), \quad \Phi_{0}=\mathcal{I},
$$
where exp denotes the exponential map in $\mathbb{H}^{2}$. By $\|N_{1}\|^{2}=\frac{1+C}{2}$
and $\|N_{2}\|^{2}=\frac{1-C}{2}$, we can write $\Phi_{l}(p,q)=\left(p_l, q_l\right)$ as follows:
$$
\begin{aligned}
&p_l=\cosh \left(C^{+} l\right) p+\tfrac{1}{C^+}\sinh \left(C^{+} l\right) N_{1},\\
&q_l=\cosh \left(C^{-} l\right) q+\tfrac{1}{C^-}\sinh \left(C^{-} l\right) N_{2},
\end{aligned}
$$
where $C^{+}=\sqrt{\tfrac{1+C}{2}}$ and $C^{-}=\sqrt{\tfrac{1-C}{2}}$.
Then, by a straightforward computation we obtain the unit normal vector field
$N^l=\left(N_{1}^l, N_{2}^l\right)$ of hypersurface $\Phi_l(\Omega)$  given by
$$
\begin{aligned}
&N_{1}^l=\cosh \left(C^{+} l\right) N_{1}+C^{+} \sinh \left(C^{+} l\right) p, \\
&N_{2}^l=\cosh \left(C^{-} l\right) N_{2}+C^{-} \sinh \left(C^{-} l\right) q.
\end{aligned}
$$
Thus it is easy to see that
$$
C_l=C, \quad J_{1} N^l=J_{1} N, \quad J_{2} N^l=J_{2} N, \quad \forall\ l \in(-\epsilon, \epsilon).
$$
Now, we can choose the orthonormal frame $\{E_{i}^l\}_{i=1}^3$ on $\Phi_l(\Omega)$ where
$$
E_{1}^l=\frac{V_l}{\sqrt{1-C^{2}}}, \quad E_{2}^l=\frac{J_{1} N^l+J_{2} N^l}{\sqrt{2(1+C)}},
\quad E_{3}^l=\frac{J_{1} N^l-J_{2} N^l}{\sqrt{2(1-C)}} .
$$
If we denote $E_{i}^{0}=E_{i}$, it follows $E_{2}^l=E_{2}$ and $E_{3}^l=E_{3}$ for any
$l \in(-\epsilon, \epsilon)$.
Denote $A_{ij}=\langle AE_i, E_j\rangle$ the components of shape operator $A$
associated to the unit normal vector field $N$ on $\Omega$.

Then, by a standard and straightforward computation of Jacobi field theory
(cf. sect. 10.2.1 of \cite{B-C-O}), we obtain that
$$
\begin{aligned}
(\Phi_l)_*{E_i}=&\left(\delta_{1 i}-l A_{1 i}\right) E_{1}^l+\left(\delta_{2 i}
\cosh \left(C^{+} l\right)-A_{2 i} \frac{\sinh \left(C^{+} l\right)}{C^{+}}\right) E_{2}^l \\
&+\left(\delta_{3 i} \cosh \left(C^{-} l\right)-A_{3 i}
\frac{\sinh \left(C^{-} l\right)}{C^{-}}\right) E_{3}^l.
\end{aligned}
$$
It is equivalent to
\begin{equation}\label{eqn:5.1}
\left(
  \begin{array}{c}
    (\Phi_l)_*{E_1} \\
    (\Phi_l)_*{E_2} \\
    (\Phi_l)_*{E_3} \\
  \end{array}
\right)=
(Q_{ij})\left(
  \begin{array}{c}
    E_1^l \\
    E_2^l \\
    E_3^l \\
  \end{array}
\right),
\end{equation}
where
$$
(Q_{ij})=
\left(
\begin{array}{ccc}
1-lA_{11} & -A_{12}\frac{\sinh(C^+ l)}{C^+} & -A_{13}\frac{\sinh(C^- l)}{C^-} \\
-lA_{12} & \cosh(C^+ l)-A_{22}\frac{\sinh(C^+ l)}{C^+} & -A_{23}\frac{\sinh(C^- l)}{C^-} \\
-lA_{13} & -A_{23}\frac{\sinh(C^+ l)}{C^+} & \cosh(C^- l)-A_{33}\frac{\sinh(C^- l)}{C^-} \\
\end{array}
\right).
$$

Furthermore, by applying Jacobi field theory again, we can obtain
\begin{equation}\label{eqn:5.2}
\left(
  \begin{array}{c}
    A_l((\Phi_l)_*{E_1}) \\
    A_l((\Phi_l)_*{E_2}) \\
    A_l((\Phi_l)_*{E_3}) \\
  \end{array}
\right)=
-(Q_{ij})'
\left(
  \begin{array}{c}
    E_1^l \\
    E_2^l \\
    E_3^l \\
  \end{array}
\right),
\end{equation}
where
$$
(Q_{ij})'=
\left(
\begin{array}{ccc}
-A_{11} & -A_{12}\cosh(C^+ l) & -A_{13}\cosh(C^- l) \\
-A_{12} & C^+\sinh(C^+ l)-A_{22}\cosh(C^+ l) & -A_{23}\cosh(C^- l) \\
-A_{13} & -A_{23}\cosh(C^+ l) & C^-\sinh(C^- l)-A_{33}\cosh(C^- l) \\
\end{array}
\right),
$$
$A_l$ is the shape operator of hypersurface $\Phi_l(\Omega)$ associated to $N^l$,
and $'$ stands for derivative with respect to $l$.

Let $Q$ denote the matrix $\left(Q_{i j}\right)$.
It follows from \eqref{eqn:5.1} and \eqref{eqn:5.2} that the mean curvatures
of nearby hypersurfaces $\Phi_l(\Omega)$ are given by
\begin{equation}\label{eqn:5.3}
H(l)=-{\rm tr}\left(Q^{-1} Q^{\prime}\right)
=-\frac{({\rm det} Q)^{\prime}}{{\rm det} Q}.
\end{equation}

\vskip 3mm

\noindent{\bf Proof of Theorem \ref{thm:1.3}}.

Now, we further assume that $M$ is an isoparametric hypersurface of
$\mathbb{H}^2\times\mathbb{H}^2$. It is well-known that
its nearby parallel hypersurfaces have constant mean curvature.

Based on the above discussion, it remains to discuss the case
when $\Omega=\{z \in M \mid C^{2}(z)<1\}$ is not empty.
Thus it follows from \eqref{eqn:5.3} that $({\rm det} Q)^{\prime}=-H(l) {\rm det} Q$
in any given connected open subset of $\Omega$, still denoted by $\Omega$.
From $Q(0)={\rm I d}$, an inductive argument shows that
$$
\left(\frac{d^{k} {\rm det} Q}{d l^{k}}\right)(0), \quad k \geq 0,
$$
are constant on $\Omega$. By \eqref{eqn:5.1}, the determinant of $Q$ is given by
$$
\begin{aligned}
\operatorname{det} Q=&(1-lA_{11})\cosh(C^- l)\cosh(C^+ l)
+(-A_{22}+lH_{12})\frac{\sinh(C^+ l)\cosh(C^- l)}{C^+}\\
&+(-A_{33}+lH_{13})\frac{\sinh(C^- l)\cosh(C^+ l)}{C^-}
+(H_{23}-lK)\frac{\sinh(C^+ l)\sinh(C^- l)}{C^-C^+},
\end{aligned}
$$
where $H_{i j}=A_{i i} A_{j j}-A_{i j}^{2}$ and $K={\rm det} A$
is the Gauss-Kronecker curvature of $\Omega$.

Now, by calculating the derivatives of the function ${\rm det} Q$ at $l=0$,
with the use of $2(H_{12}+H_{13}+H_{23})=\rho +2$, we can have

$$
\begin{aligned}
\left(\frac{d\ {\rm det}\ Q}{d l}\right)(0)&=-H,\\
\left(\frac{d^{2} {\rm det}\ Q}{d l^{2}}\right)(0)&=2(H_{12}+H_{13}+H_{23})+1=\rho+3,\\
\left(\frac{d^{4} {\rm det}\ Q}{d l^{4}}\right)(0)&=6-C^2+(4-4C)H_{12}+(4+4C)H_{13}+2\rho,\\
\left(\frac{d^{6} {\rm det}\ Q}{d l^{6}}\right)(0)&=12-5C^2+(16-12C-4C^2)H_{12}
+(16+12C-4C^2)H_{13}+(4-C^2)\rho,\\
\left(\frac{d^{8} {\rm det}\ Q}{d l^{8}}\right)(0)&=24-16C^2+C^4+(8-4C^2)\rho
+(48-32C-24C^2+8C^3)H_{12}\\
&\ \ \ \ +(48+32C-24C^2-8C^3)H_{13}.
\end{aligned}
$$

Since the left hand side of all the above equations are constant on $\Omega$,
it follows that the mean curvature $H$ and the scalar curvature $\rho$ of $\Omega$ are constant.

In the following, we discuss on the open subset
$\Omega_1=\{p \in \Omega \mid C(p) \neq 0\}\subset \Omega$.
Taking into account that $\left(\frac{d^{4} {\rm det}\ Q}{d l^{4}}\right)(0)$
and $\left(\frac{d^{6} {\rm det}\ Q}{d l^{6}}\right)(0)$ are constant, we have
$$
6-C^2+(4-4C)H_{12}+(4+4C)H_{13}+2\rho=\alpha_1,
$$
$$
12-5C^2+(16-12C-4C^2)H_{12}+(16+12C-4C^2)H_{13}+(4-C^2)\rho=\alpha_2,
$$
where $\alpha_1$ and $\alpha_2$ are two constant on $\Omega$. Then, on $\Omega_1$, we obtain
$$
H_{12}=\frac{12-6C+C^2+C^3-(4-C)\alpha_1+\alpha_2+(4-2C+C^2)\rho}{8 C\left(1-C\right)},
$$
$$
H_{13}=\frac{12+6C+C^2-C^3-(4+C)\alpha_1+\alpha_2+(4+2C+C^2)\rho}{-8 C\left(1+C\right)}.
$$
Then, substituting these two expressions into $\left(\frac{d^{8} {\rm det}\ Q}{d l^{8}}\right)(0)$,
we get
$$
12C^2-C^4-(4+2C^2)\alpha_1+4\alpha_2+4C^2\rho=\alpha_3,
$$
where $\alpha_3$ is constant on $\Omega_1$. From the fact that $\rho$ is constant on $\Omega$,
it follows that $C$ is constant on each connected component of the open set $\Omega_1$.
Thus, by the continuity of function $C$, the mean curvature $H$
and the scalar curvature $\rho$, we know that these three functions are constant on $M$.
By Remark \ref{rem:2.1} and \eqref{eqn:2.6}, $M$ also has constant principal curvatures,
then Theorem \ref{thm:1.3} follows from Theorem \ref{thm:1.1}.\qed

\section{Proofs of Theorems \ref{thm:1.4}--\ref{thm:1.6}}\label{sect:6}

\subsection{Proof of Theorem \ref{thm:1.4}}\label{sect:6.1}
In this subsection, we adopt the method in the proof to Theorem 3
of \cite{Ur} and
assume that $M$ has at most two distinct constant
principal curvatures. If $C^2=1$ holds on $M$, then by
Lemma \ref{lemma:3.3},
$M$ is an open part of $M_\Gamma$, where $\Gamma$ is a curve of $\mathbb{H}^2$
with constant curvature.  Since we focus on  local geometry here,
we assume that $C^2\neq 1$ holds on $M$ in the following.
	
First, we assume $M$ has only one constant principal curvature, saying $\lambda$.
Then, it holds $AX=\lambda X$ for any $X\in TM$. By Codazzi equation \eqref{eqn:2.3}, we have
\begin{equation}\label{eqn:6.1}
0=-\tfrac{1}{2}[\langle X,V\rangle TY-\langle Y,V\rangle TX],\ \ \forall\ X,Y\in TM.
\end{equation}
Then, taking $(X,Y)=(J_1N+J_2N,V)$ in \eqref{eqn:6.1},
we get $\tfrac{1-C^2}{2}(J_1N+J_2N)=0$, which contradicts $C^2\neq 1$.
	
Next, we suppose that $M$ has two distinct constant principal curvatures.
Let $\lambda_{1}$ and $\lambda_{2}$ be the corresponding distinct principal curvatures
with  multiplicity one and two, respectively.
Let $E_{1}$ be a unit vector field on $M$ such that $A E_{1}=\lambda_{1} E_{1}$.
Then, by using \eqref{eqn:2.5}, we have
\begin{equation}\label{eqn:6.2}
{\rm Ric}\left(E_1,E_1\right)=-\tfrac{1}{2}+2 \lambda_{1} \lambda_{2}
-\tfrac{\left\langle V, E_{1}\right\rangle^{2}}{2}
+\tfrac{C\left\langle P E_{1}, E_{1}\right\rangle}{2}.
\end{equation}
Now, the shape operator $A$ and its covariant derivative are given by
\begin{equation}\label{eqn:6.3}
\begin{aligned}
\langle AX, Y\rangle &=\lambda_{2}\langle X, Y\rangle+\left(\lambda_{1}-\lambda_{2}\right)
\left\langle X, E_{1}\right\rangle\left\langle Y, E_{1}\right\rangle ,\\
\langle(\nabla_X A)Y, Z\rangle &=\left(\lambda_{1}-\lambda_{2}\right)
\left(\left\langle Z, E_{1}\right\rangle\left\langle Y, \nabla_{X} E_{1}\right\rangle
+\left\langle Y, E_{1}\right\rangle\left\langle Z, \nabla_{X} E_{1}\right\rangle\right),
\end{aligned}
\end{equation}
where $X, Y, Z\in TM$.
	
Let $\{e_i\}_{i=1}^3$ be an orthonormal frame on $M$.
Since $M$ has constant mean curvature, by using Codazzi equation \eqref{eqn:2.3}
and \eqref{eqn:6.3}, we have
$$
\begin{aligned}
0 &=\sum_{i=1}^{3}(\nabla A)\left(X, e_{i}, e_{i}\right)
=\sum_{i=1}^{3}(\nabla A)\left(e_{i}, X, e_{i}\right)-\tfrac{1}{2} \sum_{i=1}^{3}
\left\{\langle V, X\rangle\left\langle P e_{i}, e_{i}\right\rangle
-\left\langle V, e_{i}\right\rangle\left\langle P X, e_{i}\right\rangle\right\} \\
&=\left(\lambda_{1}-\lambda_{2}\right)\left(\left\langle E_{1}, X\right\rangle {\rm div} E_{1}
+\left\langle\nabla_{E_{1}} E_{1}, X\right\rangle\right)+\tfrac{1}{2}(C\langle V, X\rangle
+\langle P V, X\rangle) \\
&=\left(\lambda_{1}-\lambda_{2}\right)\left(\left\langle E_{1}, X\right\rangle {\rm div} E_{1}
+\left\langle\nabla_{E_{1}} E_{1}, X\right\rangle\right),\ \ \ \ \forall\ X\in TM.
\end{aligned}
$$
	
Let $X=E_1$ in the above equation, and using $\|E_1\|^2$=1, it follows
\begin{equation}\label{eqn:6.4}
{\rm div} E_{1}=0 ,\quad   \nabla_{E_{1}} E_{1}=0 .
\end{equation}
Combining \eqref{eqn:6.4} with the second equation of \eqref{eqn:6.3}, we have
$$
\langle(\nabla_{E_1} A)X,Y\rangle=0,
\quad \langle(\nabla_{X} A)E_1, Y\rangle
=\left(\lambda_{1}-\lambda_{2}\right)\left\langle\nabla_{X} E_{1}, Y\right\rangle,
\ \ \forall\ X,Y\in TM.
$$
Taking $(X,Y)=(X,E_1)$ in Codazzi equation \eqref{eqn:2.3}, and using above two equations, we get
\begin{equation}\label{eqn:6.5}
\nabla_{X} E_{1}=\frac{\langle V, X\rangle T E_{1}-\left\langle V,
E_{1}\right\rangle T X}{-2\left(\lambda_{1}-\lambda_{2}\right)},\ \ \forall X\in TM.
\end{equation}
	
Then, by direct calculations, with the use of \eqref{eqn:6.4} and \eqref{eqn:6.5},
we can get another expression of ${\rm Ric}(E_1,E_1)$ as follows:

\begin{equation}\label{eqn:6.6}
{\rm Ric}\left(E_1,E_1\right)=\frac{\left\langle V,
E_{1}\right\rangle^{2}}{2\left(\lambda_{1}-\lambda_{2}\right)^{2}} .
\end{equation}
Combining \eqref{eqn:6.2} with \eqref{eqn:6.6}, we get
\begin{equation}\label{eqn:6.7}
-1+4 \lambda_{1} \lambda_{2}=-C\left\langle P E_{1}, E_{1}\right\rangle
+\left(1+\frac{1}{\left(\lambda_{1}-\lambda_{2}\right)^{2}}\right)
\left\langle V, E_{1}\right\rangle^{2}.
\end{equation}

In the following, by Lemma \ref{lemma:2.1} and \eqref{eqn:6.4}, we have
\begin{equation}\label{eqn:6.8}
\begin{aligned}
&E_1C=-2 \lambda_{1}\left\langle V, E_{1}\right\rangle,
\quad E_{1}\langle P E_{1}, E_{1}\rangle=2 \lambda_{1}\left\langle V, E_{1}\right\rangle, \\
&E_{1}\langle V, E_{1}\rangle=\lambda_{1}\left(C-\left\langle P E_{1}, E_{1}\right\rangle\right).
\end{aligned}
\end{equation}
Taking derivative in \eqref{eqn:6.7} with respect to $E_{1}$ and using \eqref{eqn:6.8},
it follows
$$
\lambda_{1}\left\langle V, E_{1}\right\rangle(\left\langle P E_{1}, E_{1}\right\rangle-C)=0.
$$
	
If $\lambda_{1}=0$, then \eqref{eqn:6.7} becomes $-1=-C\left\langle P E_{1},
E_{1}\right\rangle+\left(1+\frac{1}{\lambda_{2}^{2}}\right)\left\langle V,
E_{1}\right\rangle^{2}$. From the fact that
$-1 \leq C,\left\langle P E_{1}, E_{1}\right\rangle \leq 1$, we have
$$
1 \leq 1+\left(1+\frac{1}{\lambda_{2}^{2}}\right)\left\langle V,
E_{1}\right\rangle^{2}=C\left\langle P E_{1}, E_{1}\right\rangle \leq 1,
$$
which implies that $C^2=1$. This contradicts our assumption $C^2\neq 1$.
Hence, we have $\lambda_1 \neq 0$ and
\begin{equation}\label{eqn:6.9}
\left\langle V, E_{1}\right\rangle\left(\left\langle P E_{1}, E_{1}\right\rangle-C\right)=0.
\end{equation}
Now, taking derivative of \eqref{eqn:6.9} with respect to $E_{1}$, with the use of \eqref{eqn:6.8},
we can have $4\left\langle V, E_{1}\right\rangle^{2}=\left(\left\langle P E_{1},
E_{1}\right\rangle-C\right)^{2}$. Combining this equation with \eqref{eqn:6.9}, we get
$$
\left\langle V, E_{1}\right\rangle=0, \quad\left\langle P E_{1}, E_{1}\right\rangle=C.
$$
Then, \eqref{eqn:6.7} becomes $C^{2}=1-4\lambda_1 \lambda_2$. It follows that
$C$ is constant on $M$. Thus, from Theorem \ref{thm:1.1}, we know that
$M$ is an open part of $M_{1,-1}^{1/2}$.

In conclusion, we have completed the proof of Theorem \ref{thm:1.4}.

\vskip 3mm

\subsection{Hypersurfaces with three distinct constant principal curvatures}\label{sect:6.2}
In this subsection, let $M$ be a hypersurface of
$\mathbb{H}^2\times\mathbb{H}^2$ with three
distinct constant principal curvatures $\lambda_1$, $\lambda_2$ and $\lambda_3$.
Then there exists a local orthonormal frame $\{X_1,X_2,X_3\}$ of $M$ satisfying
\begin{equation}\label{eqn:6.10}
AX_1=\lambda_1X_1,\ AX_2=\lambda_2X_2,\ AX_3=\lambda_3X_3.
\end{equation}

Denote $A_{ij}=\langle AX_i,X_j\rangle$, $P_{ij}=\langle PX_i,X_j\rangle$ and
$b_i=\langle PX_i,N\rangle$ for $1\le i,j\le 3$.
Put $\nabla_{X_i}X_j=\sum \Gamma_{ij}^{k}X_{k}$ with
$\Gamma_{ij}^{k}=-\Gamma_{ik}^{j}$, $1\le i,j,k\le 3$.

Now, by applying the Codazzi equation \eqref{eqn:2.3} directly,
and using the fact that $P$ is a symmetry operator and it satisfies $P^2=\mathrm{Id}$,
we can have the following lemma.

\begin{lemma}\label{lemma:6.1}
The connections with respect to the frame $\{X_1,X_2,X_3\}$ satisfy the following properties:
\begin{enumerate}
\item[(1)] $\Gamma_{ij}^k+\Gamma_{ik}^j=0,\ \ 1\leq i,j,k\leq3$.
\item[(2)] $\langle(\nabla_{E_i} A){E_j},E_k\rangle=(\lambda_j-\lambda_k)\Gamma_{ij}^k,
\ \ 1\leq i,j,k\leq3$.
\item[(3)] $(\lambda_k-\lambda_j)\Gamma_{ij}^k-(\lambda_k-\lambda_i)\Gamma_{ji}^k
=-\tfrac{1}{2}(b_jP_{ik}-b_iP_{jk}),\ \ 1\leq i,j,k\leq3$.
\item[(4)] $\Gamma_{ii}^j=\tfrac{b_iP_{ij}-b_jP_{ii}}{-2(\lambda_i-\lambda_j)},
\ \ 1\leq i,j\leq3,\ i\neq j$.
\item[(5)] $(\lambda_i-\lambda_j)\Gamma_{ii}^j=-(\lambda_k-\lambda_j)\Gamma_{kk}^j,
\ \ 1\leq i,j\leq3,\ i\neq j\neq k$.
\end{enumerate}
\end{lemma}

According to Lemma \ref{lemma:3.3}, hypersurfaces of $\mathbb{H}^{2}\times$ $\mathbb{H}^{2}$
with constant $C^2=1$ have at most two distinct principal curvatures.
We assume the product angle function $C\neq\pm1$ on $M$,
which ensures that the vector field $V$ is nonzero on $M$.

\vskip 3mm

\noindent{\bf Proof of Theorem \ref{thm:1.5}}.

Let $M$ be a hypersurface of $\mathbb{H}^{2}\times\mathbb{H}^{2}$
with three distinct constant principal curvatures, and $V$ is a principal curvature
vector field on $M$. If $M$ has constant $C$, then by Theorem \ref{thm:1.1}, the proof is done.
Thus, in the following, we assume that $C$ is not constant and $C\neq\pm1$ on $M$.
We will prove that such assumption does not occur on $M$.

Due that $V$ is a principal curvature vector field, we can further choose
$X_3=\tfrac{V}{\sqrt{1-C^2}}$ in frame $\{X_i\}_{i=1}^3$ described as \eqref{eqn:6.10}.
Then, it follows that
\begin{equation}\label{eqn:6.11}
P_{13}=P_{23}=b_1=b_2=0,\  P_{33}=-C,\  b_3=\sqrt{1-C^2},\  P_{22}=-P_{11}.
\end{equation}

By using Lemma \ref{lemma:6.1} and \eqref{eqn:6.11}, we  obtain
\begin{equation}\label{eqn:6.12}
\begin{aligned}
&\Gamma_{11}^2=\Gamma_{21}^2=\Gamma_{31}^3=\Gamma_{32}^3=0,\
\Gamma_{11}^3=\frac{\sqrt{1-C^2}P_{11}}{2(\lambda_1-\lambda_3)},\
\Gamma_{22}^3=\frac{\sqrt{1-C^2}P_{22}}{2(\lambda_2-\lambda_3)},\\
&\Gamma_{12}^3=\frac{(\lambda_1-\lambda_3)\Gamma_{21}^3}{\lambda_2-\lambda_3},\
\Gamma_{31}^2=\frac{\sqrt{1-C^2}P_{12}-2(\lambda_1-\lambda_3)\Gamma_{21}^3}{-2(\lambda_1-\lambda_2)}.
\end{aligned}
\end{equation}

Now, making use of \eqref{eqn:6.11} and \eqref{eqn:6.12}, we calculate
the right hand side of $0=\langle(\bar{\nabla}_{X_1} P)N,X_2\rangle$ by definition, and obtain
\begin{equation}\label{eqn:6.13}
\Gamma_{21}^3=\frac{P_{12}\lambda_1(\lambda_2-\lambda_3)}{\sqrt{1-C^2}(\lambda_1-\lambda_3)}.
\end{equation}
Similarly, it follows from $0=\langle(\bar{\nabla}_{X_2} P)X_1,N\rangle$
and \eqref{eqn:6.11}--\eqref{eqn:6.13} that
 $P_{12}(\lambda_1-\lambda_2)\lambda_3=0$.

If $\lambda_3=0$, then by $AE_3=\lambda_3E_3$ and Lemma \ref{lemma:2.1},
it follows that $C$ is constant, which contradicts the assumption
that $C$ is not constant on $M$. Thus, we have $P_{12}=0$.

By $P_{11}^2+P_{12}^2=1$, it follows $P_{11}=\pm1$. Next, with the help of
\eqref{eqn:6.11}--\eqref{eqn:6.13} and $P_{12}=0$, a direct calculation on the right hand side of
$0=\langle(\bar{\nabla}_{X_1} P)X_1,X_3\rangle$ leads to
$C=\frac{P_{11}^2+2\lambda_1(\lambda_3-\lambda_1)}{-P_{11}}$. It implies
that $C$ is constant, which is also a contradiction.
We then have completed the proof of Theorem \ref{thm:1.5}.
\qed

\vskip 3mm

\noindent{\bf Proof of Theorem \ref{thm:1.6}}.

Let $M$ be a hypersurface of $\mathbb{H}^{2} \times \mathbb{H}^{2}$ with three distinct
constant principal curvatures and Gauss-Kronecker curvature $K=0$.
We also assume that $V$ has nonzero components in at most two of eigenspaces of shape operator $A$.
Based on Theorem \ref{thm:1.5},
we only need to consider the case when  that $V$ has
nonzero components in exact two eigenspaces of $A$ on $M$.
We will prove that such case does not occur.

If $C$ is constant in a neighborhood, then by Remark \ref{rem:2.1}, it follows that
$V$ is a principal curvature vector field, which contradicts the assumption.
Thus we assume that $C$ is not constant on $M$.

Let $\{X_i\}_{i=1}^3$ be the local orthonormal frame described as \eqref{eqn:6.10}.
We also choose the local orthonormal frame $\{e_1=\tfrac{J_1N+J_2N}{\sqrt{2(1+C)}},
e_2=\tfrac{J_1N-J_2N}{\sqrt{2(1-C)}}, e_3=\frac{V}{\sqrt{1-C^2}}\}$ of $M$, such that
\begin{equation}\label{eqn:6.14}
Pe_1=e_1,\ Pe_2=-e_2,\ Pe_3=-Ce_3+\sqrt{1-C^2}N,\ PN=CN+\sqrt{1-C^2}e_3.
\end{equation}
Then, by our assumption, without loss of generality, we can assume
$e_3=\cos\theta X_1+\sin\theta X_2$,
where $\theta\in(0,\frac{\pi}{2})$ is a function. Up to a sign of $e_1, e_2$, we can further assume
\begin{equation}\label{eqn:6.15}
e_1=\cos\beta(-\sin\theta X_1+\cos\theta X_2)+\sin\beta X_3,\
e_2=-\sin\beta(-\sin\theta X_1+\cos\theta X_2)+\cos\beta X_3,
\end{equation}
where $\beta$ is also a function.

From $e_3=\cos\theta X_1+\sin\theta X_2$ and \eqref{eqn:6.14}--\eqref{eqn:6.15}, we obtain that
\begin{equation}\label{eqn:6.16}
X_1C=-2\sqrt{1-C^2}\lambda_1\cos\theta,\ \ X_2C=-2\sqrt{1-C^2}\lambda_2\sin\theta,\ \
X_3C=0,
\end{equation}
and
\begin{equation}\label{eqn:6.17}
P\left(
   \begin{array}{c}
     X_1 \\
     X_2 \\
     X_3 \\
     N \\
   \end{array}
 \right)
=(D_{ij})
\left(
   \begin{array}{c}
     X_1 \\
     X_2 \\
     X_3 \\
     N \\
   \end{array}
 \right),
\end{equation}
where
\begin{equation*}
(D_{ij})=
\left(
  \begin{array}{cccc}
    \cos(2\beta)\sin^2\theta-C\cos^2\theta & -\tfrac{\sin(2\theta)}{2}(C+\cos(2\beta)) &
    -\sin(2\beta)\sin\theta & \sqrt{1-C^2}\cos\theta \\
    -\tfrac{\sin(2\theta)}{2}(C+\cos(2\beta)) & \cos(2\beta)\cos^2\theta-C\sin^2\theta &
    \cos\theta\sin(2\beta) & \sqrt{1-C^2}\sin\theta \\
    -\sin(2\beta)\sin\theta & \cos\theta\sin(2\beta) & -\cos(2\beta) & 0 \\
    \sqrt{1-C^2}\cos\theta & \sqrt{1-C^2}\sin\theta & 0 & C \\
  \end{array}
\right).
\end{equation*}

Put $\nabla_{X_i}X_j=\sum \Gamma_{ij}^{k}X_{k}$
with $\Gamma_{ij}^{k}=-\Gamma_{ik}^{j}$, $1\le i,j,k\le 3$.
By using Lemma \ref{lemma:6.1}, with the use of \eqref{eqn:6.17}, we can obtain
\begin{equation}\label{eqn:6.18}
\begin{aligned}
&\Gamma_{11}^2=\tfrac{\sqrt{1-C^2}\cos(2\beta)\sin\theta}{2(\lambda_1-\lambda_2)},
\ \ \ \Gamma_{21}^2=\tfrac{\sqrt{1-C^2}\cos(2\beta)\cos\theta}{2(\lambda_1-\lambda_2)},
\ \ \ \Gamma_{31}^3=-\tfrac{\sqrt{1-C^2}\cos(2\beta)\cos\theta}{2(\lambda_1-\lambda_3)},\\
&\Gamma_{32}^3=-\tfrac{\sqrt{1-C^2}\cos(2\beta)\sin\theta}{2(\lambda_2-\lambda_3)},
\ \Gamma_{11}^3=\tfrac{\sqrt{1-C^2}\sin(2\beta)\sin(2\theta)}{4(\lambda_1-\lambda_3)},
\ \Gamma_{22}^3=-\tfrac{\sqrt{1-C^2}\sin(2\beta)\sin(2\theta)}{4(\lambda_2-\lambda_3)},\\
&\Gamma_{31}^2=\tfrac{2\Gamma_{12}^3(\lambda_2-\lambda_3)
+\sqrt{1-C^2}\cos^2\theta\sin(2\beta)}{2(\lambda_1-\lambda_2)},\ \ \
\Gamma_{21}^3=\tfrac{2\Gamma_{12}^3(\lambda_2-\lambda_3)
+\sqrt{1-C^2}\sin(2\beta)}{2(\lambda_1-\lambda_3)}.
\end{aligned}
\end{equation}

Due to the Gauss-Kronecker curvature $K=0$, there is a principal curvature being $0$.
Without loss of generality,  we divide our proof into two cases
depending on values of the constant $\lambda_1$ and $\lambda_3$. By using $\bar{\nabla}P=0$,
we will prove that both two cases can not happen.

{\bf Case-i}: $\lambda_3=0$.

In this case, we begin with looking at $0=\langle(\bar{\nabla}_{X_1} P)N,X_3\rangle$.
Calculating the right side by definition, and using \eqref{eqn:6.16}--\eqref{eqn:6.18}, we get
$$
(\sqrt{1-C^2}\Gamma_{12}^3\lambda_1-\cos\beta(2\lambda_1^2+
(C^2-1)\cos^2\theta)\sin\beta)\sin\theta=0.
$$
It implies that
\begin{equation}\label{eqn:6.19}
\Gamma_{12}^3=\frac{(2\lambda_1^2+(C^2-1)\cos^2\theta)\sin(2\beta)}{2\sqrt{1-C^2}\lambda_1}.
\end{equation}
Then, from $0=\langle(\bar{\nabla}_{X_1} P)X_3,X_3\rangle$
and $0=\langle(\bar{\nabla}_{X_1} P)N,X_2\rangle$,
by using of \eqref{eqn:6.16}--\eqref{eqn:6.19}, we get
\begin{equation}\label{eqn:6.20}
\begin{aligned}
X_1\theta=\tfrac{(2C\lambda_1(\lambda_2-\lambda_1)+(C^2-1+2\lambda_1(\lambda_1-\lambda_2))
\cos(2\beta))\sin\theta}{2\sqrt{1-C^2}(\lambda_1-\lambda_2)},\
X_1\beta=\tfrac{(C^2-1+2\lambda_1^2)\cos\theta\sin(2\beta)}{-2\sqrt{1-C^2}\lambda_1}.
\end{aligned}
\end{equation}

Similarly, calculating $0=\langle(\bar{\nabla}_{X_2} P)X_3,X_3\rangle$,
$0=\langle(\bar{\nabla}_{X_2} P)N,X_2\rangle$
and $0=\langle(\bar{\nabla}_{X_2} P)N,X_3\rangle$, by using \eqref{eqn:6.16}--\eqref{eqn:6.18},
we obtain
\begin{equation}\label{eqn:6.21}
\begin{aligned}
X_2\theta=\tfrac{(2C\lambda_2(\lambda_1-\lambda_2)+(C^2-1-2\lambda_2(\lambda_1-\lambda_2))
\cos(2\beta))\cos\theta}{2\sqrt{1-C^2}(\lambda_1-\lambda_2)},\
X_2\beta=\tfrac{(C^2-1+2\lambda_2^2)\sin\theta\sin(2\beta)}{-2\sqrt{1-C^2}\lambda_2},
\end{aligned}
\end{equation}
and another expression of $\Gamma_{12}^3$ as follows:
\begin{equation}\label{eqn:6.22}
\Gamma_{12}^3=\frac{(4\lambda_1\lambda_2^2-(C^2-1)(2\lambda_2-\lambda_1
+\lambda_1\cos(2\theta))) \sin(2\beta)}{-4\sqrt{1-C^2}\lambda_2^2}.
\end{equation}

Next, using \eqref{eqn:6.16}--\eqref{eqn:6.18} to $0=\langle(\bar{\nabla}_{X_3} P)N,X_3\rangle$,
we get
\begin{equation}\label{eqn:6.23}
\cos(2\beta)(\lambda_1+\lambda_2+(\lambda_2-\lambda_1)\cos(2\theta))=0.
\end{equation}
If $\cos(2\beta)=0$ holds locally, then by the second equation in \eqref{eqn:6.20}
and $\theta\in(0,\frac{\pi}{2})$, we get $C^2-1+2\lambda_1^2=0$,
which means that $C$ is constant locally. It is a contradiction.

So, we assume that $\cos(2\beta)\neq 0$ locally. Then, by \eqref{eqn:6.23}, we have
$\cos(2\theta)=\frac{\lambda_1+\lambda_2}{\lambda_1-\lambda_2}$,
which means that $\theta$ is constant. Then, by the first equation in \eqref{eqn:6.20}, we have
\begin{equation}\label{eqn:6.24}
\cos(2\beta)=\frac{2C\lambda_1(\lambda_1-\lambda_2)}{C^2-1+2\lambda_1(\lambda_1-\lambda_2)},
\end{equation}
where $C^2-1+2\lambda_1(\lambda_1-\lambda_2)\neq0$ because of $C$ is not constant.

Substituting \eqref{eqn:6.24} into the first equation in \eqref{eqn:6.21},
and using the fact that $C$ is not constant, we have $\lambda_1+\lambda_2=0$.
Then, $\cos(2\theta)=\frac{\lambda_1+\lambda_2}{\lambda_1-\lambda_2}=0$,
which implies that $\theta=\frac{\pi}{4}$.

Since $C$ is not constant, from \eqref{eqn:6.24},
we can assume $\sin(2\beta)\neq0$ in a neighborhood $U$.
Now, comparing \eqref{eqn:6.19} with \eqref{eqn:6.22}, and using $\sin(2\beta)\neq0$,
$\lambda_1+\lambda_2=0$ and $\theta=\frac{\pi}{4}$, we can obtain
$$
C^2-1+2\lambda_1^2=0.
$$
It means that $C$ is constant in $U$, which is a contradiction.
Thus, we have proved that {\bf Case-i} does not occur.

{\bf Case-ii}: $\lambda_1=0$.

{\bf Claim}: $\sin(2\beta)\neq 0$ on $M$.

To verify this claim we argue by contradiction. Suppose $\sin(2\beta)=0$ holds locally,
then $\cos(2\beta)=\pm1$.

If $\cos(2\beta)=1$, applying \eqref{eqn:6.16}--\eqref{eqn:6.18} to
$0=\langle(\bar{\nabla}_{X_3} P)X_3,X_1\rangle$ and $0=\langle(\bar{\nabla}_{X_3} P)N,X_3\rangle$,
we have
$$
\begin{aligned}
&-2\lambda_3(-1+C+2\lambda_3^2)+\lambda_2(-3+C+4\lambda_3^2)+(1+C)\lambda_2\cos(2\theta)=0,\\
&-2\lambda_3(-1+C+2\lambda_3^2)+\lambda_2(-1+C+4\lambda_3^2)+(-1+C)\lambda_2\cos(2\theta)=0.
\end{aligned}
$$
By direct calculations, the above two equations imply that $\cos(2\theta)=1$ and $C=1-2\lambda_3^2$,
which means that $C$ is constant. It is a contradiction.

If $\cos(2\beta)=-1$, applying \eqref{eqn:6.16}--\eqref{eqn:6.18} to
$0=\langle(\bar{\nabla}_{X_3} P)X_3,X_1\rangle$ and $0=\langle(\bar{\nabla}_{X_3} P)N,X_3\rangle$,
we have
$$
\begin{aligned}
&\lambda_2(3+C-4\lambda_3^2)-2\lambda_3(1+C-2\lambda_3^2)+(-1+C)\lambda_2\cos(2\theta)=0,\\
&(1+C)(1+\cos(2\theta))\lambda_2-2(1+C)\lambda_3-4\lambda_2\lambda_3^2+4\lambda_3^3=0.
\end{aligned}
$$
After a straightforward  calculation it leads to
$\cos(2\theta)=1$ and $C=-1+2\lambda_3^2$.
Then $C$ is also constant, which is a contradiction.
We have completed the proof of the Claim.

\vskip 1mm

Now, we consider $0=\langle(\bar{\nabla}_{X_1} P)N,X_3\rangle$.
Calculating the right side by definition, and using \eqref{eqn:6.16}--\eqref{eqn:6.18},
we get
$$
\sqrt{1-C^2}\Gamma_{12}^3\lambda_3+(C^2-1)\cos^2\theta\sin\beta\cos\beta=0,
$$
which implies that
\begin{equation}\label{eqn:6.25}
\Gamma_{12}^3=\frac{\sqrt{1-C^2}\cos^2\theta\sin(2\beta)}{2\lambda_3}.
\end{equation}
Then, making use of $\sin(2\beta)\neq0$, \eqref{eqn:6.16}--\eqref{eqn:6.18}
and \eqref{eqn:6.25} into $0=\langle(\bar{\nabla}_{X_1} P)X_3,X_3\rangle$
and $0=\langle(\bar{\nabla}_{X_1} P)N,X_2\rangle$,
we get
\begin{equation}\label{eqn:6.26}
\begin{aligned}
X_1\theta=\frac{\sqrt{1-C^2}\sin\theta\cos(2\beta)}{2\lambda_2},\
X_1\beta=\frac{\sqrt{1-C^2}\cos\theta\sin(2\beta)}{-2\lambda_3}.
\end{aligned}
\end{equation}

If $\cos(2\beta)=0$ holds locally, then it contradicts $\sin(2\beta)\neq 0$ and
the second equation in \eqref{eqn:6.26}. Thus, without loss of generality,
we also assume that $\cos(2\beta)\neq0$ on $M$.

Now, calculating $0=\langle(\bar{\nabla}_{X_2} P)N,X_3\rangle$,
by using \eqref{eqn:6.16}--\eqref{eqn:6.18}
we get another expression of $\Gamma_{12}^3$ as follows:
\begin{equation}\label{eqn:6.27}
\Gamma_{12}^3=\frac{(\lambda_3-C^2\lambda_3+4\lambda_2^2\lambda_3+2\lambda_2(C^2-1-2\lambda_3^2)
-(C^2-1)\lambda_3\cos(2\theta))\sin(2\beta)}{4\sqrt{1-C^2}(\lambda_2-\lambda_3)^2}.
\end{equation}

By carrying out the similar calculation on $0=\langle(\bar{\nabla}_{X_3} P)X_3,X_3\rangle$,
$0=\langle(\bar{\nabla}_{X_3} P)N,X_2\rangle$ and
$0=\langle(\bar{\nabla}_{X_3} P)N,X_3\rangle$,
we get
\begin{equation}\label{eqn:6.28}
X_3\beta=\frac{\sqrt{1-C^2}\lambda_2\cos(2\beta)\sin(2\theta)}{4\lambda_3(\lambda_2-\lambda_3)},
\end{equation}
\begin{equation}\label{eqn:6.29}
\sqrt{1-C^2}\Gamma_{12}^3(\lambda_3-\lambda_2)+\sqrt{1-C^2}\lambda_2(X_3\theta)+
(C^2-1)\cos\beta\cos^2\theta\sin\beta+\lambda_2\lambda_3\sin(2\beta)=0,
\end{equation}
\begin{equation}\label{eqn:6.30}
4C(\lambda_2-\lambda_3)\lambda_3^2+\cos(2\beta)(-2\lambda_3(C^2-1+2\lambda_3^2)+\lambda_2
(C^2-1+4\lambda_3^2)+(C^2-1)\lambda_2\cos(2\theta))=0.
\end{equation}

Taking derivative of \eqref{eqn:6.30} with respect to $X_3$, and using
\eqref{eqn:6.16}, \eqref{eqn:6.28} and $\cos(2\beta)\neq0$, we get
$$
X_3\theta=\frac{(-2\lambda_3(-1+C^2+2\lambda_3^2)+\lambda_2(-1+C^2+4\lambda_3^2)
+(-1+C^2)\lambda_2 \cos(2\theta))\sin(2\beta)}{4\sqrt{1-C^2}(\lambda_2-\lambda_3)\lambda_3}.
$$
Substituting above expression into \eqref{eqn:6.29}, it follows
\begin{equation}\label{eqn:6.31}
\Gamma_{12}^3=\frac{(-(C^2-1)\lambda_3^2+\lambda_2(C^2-1+8\lambda_3^2)(\lambda_2-\lambda_3)
+(C^2-1)(\lambda_2^2+\lambda_2\lambda_3-\lambda_3^2)\cos(2\theta))
\sin(2\beta)}{4\sqrt{1-C^2}(\lambda_2-\lambda_3)^2\lambda_3}.
\end{equation}

Now, comparing the expressions \eqref{eqn:6.25} and \eqref{eqn:6.27},
as well as \eqref{eqn:6.25} and \eqref{eqn:6.31},
and using $\sin(2\beta)\neq0$, we can obtain
\begin{equation}\label{eqn:6.32}
\begin{aligned}
&(C^2-1)(\lambda_2-2\lambda_3)\cos(2\theta)=4\lambda_3^3-\lambda_2(C^2-1+4\lambda_3^2),\\
&(C^2-1)(\lambda_3-2\lambda_2)\cos(2\theta)=\lambda_3(3-3C^2-8\lambda_3^2)
+2\lambda_2(C^2-1+4 \lambda_3^2).
\end{aligned}
\end{equation}
From \eqref{eqn:6.32}, we have $-1+C^2+2\lambda_3^2=0$,
which means that $C$ is constant. It is also a contradiction.
Thus, {\bf Case-ii} does not occur.
We have completed the proof of Theorem \ref{thm:1.6}.
\qed

\begin{acknow}
This is partially supported by National Natural Science Foundation of China
(Grant Nos.~11831005, 12061131014 and 12171437) and China Postdoctoral Science Foundation
(No.2022M721871). The authors would like to thank Professor Haizhong Li and Professor Chao Qian
for their helpful conversations on this work. They are also
indebted to the referees for their valuable comments and suggestions
that greatly improved the presentation.
\end{acknow}

%%%%%%%%%%%%%%%%%%%%%%%%%%%%%%%%%%%%%%%%%%%%%%%%%%%%%%%%%%%%%%%%%%%%

\end{document}